\documentclass[leqno,12pt]{amsart}
\usepackage{amssymb,amsmath, amsthm,mathrsfs,
enumitem, ulem,cancel,comment}
\usepackage[at]{easylist}
\usepackage{a4wide,mathscinet,amsxtra}
\usepackage{tikz}
\usetikzlibrary{patterns}
\usepackage{color}
\usepackage{xcolor}
\pagecolor{white}
\usepackage{hyperref}

\usepackage{float}
\usepackage{relsize}

\theoremstyle{plain}

\newtheorem{lemma}{Lemma}[section]
\newtheorem{theorem}[lemma]{Theorem}
\newtheorem{proposition}[lemma]{Proposition}
\newtheorem{corollary}[lemma]{Corollary}

\newtheorem{remark}[lemma]{Remark}

\newtheorem{definition}[lemma]{Definition}

\newtheorem{conjecture}[lemma]{Conjecture}

\newcommand\Z{\mathbb Z}
\newcommand\R{\mathbb R}
\newcommand\C{\mathbb C}
\newcommand\e{\epsilon}
\newcommand{\Irr}{\operatorname{Irr}}
\newcommand{\GL}{\operatorname{GL}}
\newcommand{\Sp}{\operatorname{Sp}}
\newcommand{\SO}{\operatorname{SO}}
\newcommand{\Cusp}{\mathcal{C}}
\newcommand{\Disc}{\mathcal{D}}

\newcommand{\Irrcl}{\Irr^{cl}}

\newcommand{\Tempcl}{\mathcal{T}^{cl}}
\newcommand{\spsi}{\operatorname{s.p.}}

\newcommand{\ssm}{\operatorname{s.s.}}

\newcommand\mc{\mathcal}

\newcommand\ppe{\pi(\psi,\epsilon)}

\newcommand{\Jord}{\operatorname{Jord}}

\newcommand\jrp{\Jord_\rho(\psi)}

\newcommand\jrpd{\Jord_\rho(\psi_d)}

\newcommand\jp{\text{\rm Jord}(\psi)}

\newcommand\jr{\Jord_\rho}

\newcommand{\Jac}{\operatorname{Jac}}

\newcommand\slc{SL(2,\mathbb C)}

\newcommand\wddg{W_F\times SL(2,\mathbb C) \times SL(2,\mathbb C)}

\newcommand\rab{(\rho,a,b)}
\newcommand\rabp{(\rho,a',b')}
\newcommand\rabpp{(\rho,a'',b'')}

\newcommand{\Pgp}{\Psi_{\operatorname{g.p.}}}
\newcommand{\Pe}{\Psi_{\operatorname{ele.}}}
\newcommand{\Pddr}{\Psi_{\operatorname{d.d.r.}}}

\newcommand\h{\hookrightarrow}
\newcommand\hui{\underset{\text{u.i.sub.}}\hookrightarrow}

\newcommand{\ASS}{DL{ }}

\newcommand\D{\Delta}

\newcommand\s{\sigma}

\newcommand\SC{ \mathcal S(\Cusp)}

\newcommand\supp{\operatorname{supp}}

\newcommand\tr{^{\mathbf t}}

\newcommand\la{\langle}
\newcommand\ra{\rangle}

\numberwithin{equation}{section}

\begin{document}

\title[On unitarizability and Arthur packets]
{On unitarizability and Arthur packets}

\mark{Marko Tadi\'c}

\author{Marko Tadi\'c}

\address{Department of Mathematics, University of Zagreb
\\
Bijeni\v{c}ka 30, 10000 Zagreb,
 Croatia\\
Email: \tt tadic{\char'100}math.hr}


\thanks{This work has been supported by Croatian Science Foundation under the
project IP-2018-01-3628.}
\date{\today}

\begin{abstract}
In this paper we begin to explore the relation between the question of unitarizability of classical $p$-adic groups, and Arthur packets, starting from \cite{MR-T-CR3}.

\end{abstract}

\maketitle

\setcounter{tocdepth}{1}
\tableofcontents

\section{Introduction}\label{intro}

In \cite{MR3969882} we proposed  a strategy to approach unitarizability of classical groups over a $p$-adic field $F$ of characteristic $0$. In that strategy, the only relevant information is the cuspidal reducibility exponent, which is an element of 
 $$
 \tfrac12\Z_{\geq0}
 $$
 (therefore, these are the parameters with which we work).
  We applied this strategy in \cite{MR-T-CR3} to classify unitarizability in coranks $\leq 3$. The key to control unitarizability in \cite{MR-T-CR3} is to understand it in the case of the so-called critical points (see Definition \ref{def-critical}). These are the places where the most important irreducible unitarizable representations show up. 
 Non-unitarizable representations also give   some key information for proving exhaustion. 
  The majority of irreducible subquotients  are non-unitarizable. Still, approximately 100 types of irreducible subquotients  are unitarizable.
Unitarizability of these  representations  was proved using standard methods of representation theory,  except in the case of the representations given in the Langlands classification by
\begin{equation}
\label{intro-Moe-r}
 L(\nu^{\alpha-1}\rho,\nu^\alpha\rho;\delta([\nu^{\alpha}\rho];\sigma)),
 \end{equation}
 where $\rho$ and $\sigma$ are  irreducible  cuspidal representations of a general linear  and a classical group, respectively, and $\alpha$ is corresponding (exceptional) reducibility point  which satisfies $\alpha \geq \frac32$
  (see \ref{LC CG} and \ref{simple SIR} for  a description of the notation).
 C. M\oe glin  proved its unitarizability using her  explicit characterisation of Arthur packets (see her Appendix A in \cite{MR-T-CR3}). This is the single place in \cite{MR-T-CR3} where Arthur packets interact (explicitly) with questions of unitarizability. 
  The lowest rank cases when  representations \eqref{intro-Moe-r} show up are $\Sp(10,F)$ and  $\SO(11,F)$.
 
 We expect that the role of Arthur packets is much deeper in the unitarizability problem. This paper is a step in trying to understand (and an attempt  to formulate more precisely) this interplay. In this paper we consider  symplectic and split     special odd-orthogonal groups over $F$ (we expect that the results of this paper also hold  for other classical groups).
 
 The  starting point of this paper in the direction of  Arthur packets   
 are the M\oe glin representations \eqref{intro-Moe-r} which we considered in \cite{MR-T-CR3}.  
  In this paper we extend the M\oe glin construction to
 a two-parameter family
 $$
\pi_{m,n}:= L([\nu^{\alpha-1}\rho],[\nu^{\alpha}\rho],\dots, [\nu^{\alpha+m}\rho];\delta([\nu^{\alpha}\rho,\nu^{\alpha+n}\rho];\sigma)), \quad m,n\geq 0
 $$
  (see Theorem \ref{theorem >1}).  We have seen in \cite{MR-T-CR3} that  the representations $\pi_{0,0}$ (i.e. \eqref{intro-Moe-r})   are isolated in the unitary duals. We expect that all the representations $\pi_{n,m}$ are isolated as well. Further, these representations  satisfy the following very simple formula for the Aubert duality
\begin{equation}
\label{intro-tr}
\pi_{m,n}\tr=\pi_{n,m}.
  \end{equation}
 Recall that for the Speh  representations we have analogous formulas for duality 
 $$
 u_\rho(m,n)\tr=u_\rho(n,m)
 $$
  (see  \ref{LC-GL} for  explanation of the notation).

        Further, in Theorems \ref{thm-0} and \ref{thm-1/2} we describe  analogous two-parameter families   in the cases of non-exceptional reducibilities $0$ and $\tfrac12$ (excluding finitely many  $\rho$'s, all reducibilities are of this type when we fix $\sigma$), and also handle the case of exceptional reducibility at $1$ (Theorem \ref{thm-1}). 
 We also compute  the  Aubert duals of these representations, and get formulas similar to  formula \eqref{intro-tr} (see Propositions \ref{tr 0 inv}, \ref{prop inv 1/2} and \ref{prop inv 1}). We also expect  for these representations to be isolated, excluding few of them (with very low indexes).

\medskip

 The second starting point of this paper in the direction of the Arthur packets 
 is the list of unitarizable irreducible subquotients at critical points in \cite{MR-T-CR3}. As we already mentioned, this is a list of around 100 types of them.
 The list (which was essentially obtained through case-by-case considerations; it took about half of \cite{MR-T-CR3} to get it) does not reveal a clear pattern; in fact, at first glance, it seems somewhat random.

 Recall that we know from our present (very limited) understanding of unitarizability: In the case of general linear groups, as well as generic and spherical representations of classical groups (see \cite{MR0870688}, \cite{MR2046512} and \cite{MR2767523}), automorphic representations play a key role in describing unitarizability.”
   It may be natural to try to see if this is also the case  for  classical groups. 

 Arthur packets provide us with a significant number of  unitarizable representations (their unitarizability  follows from the fundamental  work \cite{MR3135650} of J. Arthur). For a number  of representations of Arthur type   we do not know to prove their unitarizability  by other methods (at least not at the present time).  The most interesting representations of Arthur type seem to be those which are irreducible subquotients at  critical points.

Unitary duals carry natural topology (\cite{MR0150241}), and the most interesting representations in the unitary duals are the isolated representations. The following conjecture is an attempt to relate
    Arthur type representations, critical points  and isolated representations.
     It may easily happen that the following conjecture is not true, but still we expect that this line of thinking is useful.
  This conjecture may be very hard to prove (if it is true).

  \begin{conjecture}
 \label{intro conj-1-intro} {\ }
 
 \begin{enumerate}
 
 \item
 \label{conj A-cp}
 
Let $\pi$ be an irreducible subquotient at a critical point. Then

the following two claims are equivalent:

\begin{enumerate}

\item
\label{intro uni-c} $\pi$ is unitarizable.

\item 
\label{intro Arthur-c}
$\pi$ is a member of some Arthur packet.

\end{enumerate}

\item
\label{conj: iso-critical-intro}
 If $\pi$ is an isolated representation in the unitary dual,
then $\pi$ is a 
representation of critical type.

\item
\label{conj: iso-A-intro}
 Each isolated representation $\pi$ in the unitary dual is in an Arthur packet.

\end{enumerate}
 
 \end{conjecture}
 
 Note that  \eqref{conj A-cp} and 
 \eqref{conj: iso-critical-intro} 
 would imply 
 \eqref{conj: iso-A-intro} (claim \eqref{conj: iso-critical-intro} of the above conjecture was stated as Conjecture 8.16 in \cite{MR-T-CR3}).
 
 Now we briefly comment   
 claim \eqref{conj A-cp} of the above conjecture.
  As  we already noted,  the statement that 
  \eqref{intro Arthur-c} 
  implies 
  \eqref{intro uni-c} is known 
  (so  claim \eqref{conj A-cp}
 is that \eqref{intro uni-c} implies \eqref{intro Arthur-c}).
  This claim would give   a relatively   simple pattern of understanding unitarizability in the most delicate cases for classical groups.
  It relates  a highly mysterious and very hard question of unitarizability/non-unitarizability at the critical points to (at least a little bit) less mysterious and more combinatorial   question of belonging to  Arthur packets (a problem which seems easier to handle).

 If the above claim is true, it could also provide   upper bounds for unitarizability in general (and therefore,  be useful for exhaustion questions). Namely, the exhaustion is obtained  in \cite{MR-T-CR3} (as well as in the other papers on unitarizability in the corank two) by reducing the questions of the non-unitarizability  to the known non-unitarizability at the critical points.
 
We have the following (very limited) support to the above conjecture:
 
 \begin{theorem}
 \label{intro thm-conj}
  Conjecture
 \ref{intro conj-1-intro} holds if $\pi$ is an irreducible representation     which is a subquotient at a critical point in corank $\leq3$, or an unramified  representation, or a generic representation.
\end{theorem}

The fact that the above conjecture holds if $\pi$ is an unramified  (resp. a generic) representation follows easily from \cite{MR2767523} (resp\cite{MR2046512}). If $\pi$ is an irreducible subquotient at a critical point in corank $\leq3$, then  claims \eqref{conj: iso-critical-intro} and \eqref{conj: iso-A-intro} follow directly from \cite{MR-T-CR3}, while claim \eqref{conj A-cp}
is
 Theorem \ref{cr leq 3} of this paper.

\smallskip

Very important parts of unitary duals are automorphic duals (introduced in \cite{MR2331344} by L. Clozel; we recall  this notion in section \ref{sec: intermediate} of this paper). In the cases where unitarizability is known, automorohic duals 
usually contain   the most important representations of the unitary duals (like isolated representations). If  \eqref{conj: iso-A-intro} of  Conjecture \ref{intro conj-1-intro} holds, then each isolated representation in the unitary dual would be isolated in the automorphic dual.  
Such a representation  must be primitive, i.e. it must not be a subrepresentation of a representation parabolically induced 
by an Arthur type representation of a proper Levi subgroup (see Definition \ref{def-prim}).

The question of the isolated representations in the unitary duals is one of the most delicate problems of the unitarizability (see \cite{MR0894588} for the case of unitary duals of general linear groups and \cite{MR2767523} for the case of unramified unitary duals of classical groups).          The study of automorphic duals in \cite{MR2742213} (under assumption of the generalized Ramanujan conjecture and the  ``Arthur + $\e$"  conjecture of Clozel from \cite{MR2331344}) suggests that the lists of isolated representations in the automorphic duals could be considerably simpler than the same lists  in the unitary duals. While we do not see, at the moment, a way to conjecture much regarding a list of isolated representations in the unitary duals, we may try to guess the following qualitative characterization of the isolated representations in the automorphic duals (which could yield a quantitative description):

\medskip

\noindent
{\bf Question 1.3.}
Is each primitive representation of Arthur type isolated in the automorphic dual?

\medskip

 The key for handling Arthur packets in our paper is the M\oe glin explicit construction of Arthur packets (together with the work of B. Xu related to this). Let us note that the techniques of M\oe glin seem to fit well with the approach to the unitarizability based only on reducibility points. In construction of Arthur packets, knowledge  of the Aubert involution of representations is very useful. Some crucial  ideas for   computation of the involution belong to C. Jantzen. These are principal tools that we use in this paper.

 Recently H. Atobe and A. M\'inguez in \cite{AtMiZel} and H. Atobe in \cite{At-const} made a crucial breakthrough finding algorithms for the Aubert involution and  construction of Arthur packets.

 We are very thankful to C. M\oe glin for a series of discussions and for sharing her results with us. Discussions with E. Lapid  helped us better understand  some ideas on which this paper is based. P. Baki\'c  has read the  paper and gave us a number of useful suggestions, which helped us a lot to improve the style of the paper. We are very thankful to the referee for a very careful reading and a number of  important suggestions and corrections. Thanks to them, this paper is much easier to read.

 The structure of the paper is the following. Section \ref{sec: notation} introduces the notation of the representation theory of classical $p$-adic groups, while   section \ref{MPAP} collects the notation and   some facts about the M\oe glin construction of Arthur packets. Sections \ref{sec: >1}, \ref{sec: 0}, \ref{sec: 1/2} and \ref{sec: 1} bring the constructions of families of Arthur representations corresponding to four types of reducibility points.
 In  section \ref{sec: >1} we consider the case of reducibility $>1$, where details of all the proofs are presented.
In sections \ref{sec: 0}, \ref{sec: 1/2} and \ref{sec: 1} the cases of  reducibility points  $0$, $\tfrac12$ and $1$  are considered respectively. The proofs of the claims in   sections \ref{sec: 0}, \ref{sec: 1/2} and \ref{sec: 1} are obtained using  similar ideas and techniques  as the proofs of the corresponding claims in   section \ref{sec: >1}. Because of this, we completely omit proofs in  sections \ref{sec: 0}, \ref{sec: 1/2} and \ref{sec: 1}. In the last section we show that each irreducible unitarizable subquotient  of a critical point in corank $\leq 3$ is of Arthur class. In   the Appendix we show that some  of the simplest complementary series  can be of Arthur class when reducibility exponent is $>1$.

\section{Notation
}
\label{sec: notation}

We first recall briefly  the well-known notation  for $p$-adic general linear groups established  by J. Bernstein and A. V. Zelevinsky (\cite{MR584084}; see also \cite{MR689531}), and its natural extension to classical  $p$-adic  groups. 

A $p$-adic field $F$ of characteristic zero will be fixed. All representations  considered in this paper will be smooth. The Grothendieck group of the category $\text{Alg}_{\text{f.l.}}(G)$ of all finite length representations  of a connected reductive $p$-adic group $G$ is denoted by 
$$
\mathfrak R(G).
$$
We have a natural map $\text{s.s.}:\text{Alg}_{\text{f.l.}}(G)\rightarrow \mathfrak R(G)$ called semi-simplification. There is a natural partial order on $\mathfrak R(G)$. For two finite length representations $\pi_1$ and $\pi_2$ of $G$, the fact $\ssm(\pi_1)\leq\ssm(\pi_2)$ we will write simply as $\pi_1\leq \pi_2$.
The contragredient representation of $\pi$ will be denoted by $\tilde \pi$. We can lift the mapping $\pi\rightarrow \tilde \pi$ to an additive group homomorphism  $\sim: \mathfrak R(G) \rightarrow \mathfrak R(G)$. 

If $\Pi$ (resp. $\pi$) is a representation (resp. an   irreducible representation) of $G$ , then 
$$
\pi\hui \Pi
$$
will mean that $\pi$ is a unique   irreducible subrepresentation of $\Pi$.

\subsection{Hopf algebra for general linear groups}
The modulus character on $F$ is denoted by $| \ |_F$, and the character $|\det|_F$ of $\GL (n,F)$ by 
$
\nu.
$
Let $n=n_1+n_2, n_i\geq 0$. Denote by $P_{(n_1,n_2)}=M_{(n_1,n_2)}N_{(n_1,n_2)}$ the  parabolic subgroup of $\GL (n,F)$ which is standard with respect to the minimal parabolic subgroup of upper triangular matrices, such that    $M_{(n_1,n_2)}$ is naturally isomorphic to $\GL (n_1,F)\times \GL (n_2,F)$.
For  representations $\pi_i$, $i=1,2$, of $\GL (n_i,F)$, 
  denote
$$
\pi_1\times \pi_2:=\text{Ind}_{P_{(n_1,n_2)}}^{\GL (n,F)}(\pi_1\otimes \pi_2).
$$
Let
$
R:=\underset{n\geq0}{\oplus}\mathfrak R(\GL (n,F).
$
Then $\times $ lifts naturally to a multiplication on $R$, and in this way we get     commutative graded $\Z$-algebra structure on $R$. We can factorise 
$\times:R\times R\rightarrow R$ through
$
m:R\otimes R\rightarrow R.
$

The normalised Jacquet module with respect to $P_{(n_1,n_2)}$ of a representation $\pi$ of $\GL (n,F)$ is denoted by
$
r_{(n_1,n_2)}(\pi).
$
If $\pi$  is of finite length, then we can set
$$
m^*(\pi):=\sum_{k=0}^n \text{s.s.}(r_{(k,n-k)}(\pi))\in R\otimes R.
$$
One extends additively $m^*$ on whole $R$, and gets  graded Hopf algebra structure on $R$.

\subsection{Segments and corresponding irreducible subrepresentations}
Denote by $\mathcal C$ (resp. $\mathcal D$) the set of all equivalence classes of irreducible cuspidal (resp. essentially square integrable) representations of all $\GL (n,F)$, $n\geq1$.  
For $\delta\in \mc D$, there exists unique $e(\delta)\in \R$ and unitarizable $\delta^u\in \mc D$ such that
$$
\delta=\nu^{e(\delta)}\delta^u.
$$
For $\rho\in\mathcal C$ and $x,y\in \R$ such that $y-x\in\Z_{\geq0}$, the set
$
[\nu^x\rho,\nu^y\rho]:=\{\nu^x\rho,\nu^{x+1}\rho,\dots,\nu^y\rho\}
$
is called a segment of cuspidal representations of general linear groups (one-point segment $[\nu^x\rho,\nu^x\rho]$ will be denoted simply by $[\nu^x\rho]$).We denote it also by
$$
[x,y]^{(\rho)},
$$
or simply by $[x,y]$ when we will work with a fixed $\rho$ (usually this will be the case later). The set of all segments of cuspidal representations is be denoted by $\mathcal S(\mathcal C)$.
The representation
$\nu^x\rho\times\nu^{x+1}\rho\times\dots\times\nu^y\rho$ (resp. 
$\nu^y\rho\times\nu^{y-1}\rho\times\dots\times\nu^x\rho$) contains a unique irreducible subrepresentation which will be denoted by 
$$
\text{$\la\nu^x\rho,\nu^{x+1}\rho,\dots,\nu^y\rho\ra$ (resp. 
$\la\nu^y\rho,\nu^{y-1}\rho,\dots,\nu^x\rho\ra$).}
$$
  When we  deal with a fixed $\rho$, these representations will be denoted simply by 
$
\la x,x+1,\dots,y\ra, \ (\text{resp. } \la y,y-1,\dots,x\ra).
$
For a segment $[x,y]^{(\rho)}\in \mc S(\mc C)$ denote
$
\delta([x,y]^{(\rho)}):= \la\nu^y\rho,\nu^{y-1}\rho,\dots,\nu^x\rho\ra.
$
Then $\delta([x,y]^{(\rho)})\in \mc D$. For $n\geq1$ set
$
\delta(\rho,n):=\delta([-\tfrac{n-1}2,\tfrac{n-1}2]^{(\rho)}).
$

Let $\pi$ be an irreducible representation of a general linear group. Then there exist $\rho_1,\dots,\rho_k\in \mc C$ such that $\pi\h  \rho_1\times\dots\times\rho_k$. The multiset $(\rho_1,\dots,\rho_k)$ is called the (cuspidal) support of $\pi$, and  is denoted by $\supp(\pi)$.

\subsection{Langlands classification for general linear groups}
\label{LC-GL}
For a set $X$, denote by $M(X)$ the set of all finite multisets in $X$. For $d=(\delta_1,\dots,\delta_k)\in M(\mc D)$ chose a permutation $p$ of $\{1,\dots,k\}$ such that $e(\delta_{p(1)})\geq \dots \geq e(\delta_{p(k)})$. Then the representation
$
\lambda(d):= \delta_{p(1)}\times \dots \times \delta_{p(k)}
$
has a unique irreducible quotient, denoted by
$
L(d).
$
This defines a bijection  from $M(\mc D)$ onto the set of equivalence classes of irreducible representations of all groups $\GL (n,F)$, $n\geq 0$ (Langlands classification). 
Another way to express this classification is by $M(\mc S(\mc C))$. To $a=(\Delta_1,\dots,\Delta_k)\in M(\mc S(\mc C))$ attach
$$
L(a):=L(\delta(\Delta_1), \dots ,\delta(\Delta_k)).
$$
This is the version of the Langlands classification  which  we will  use in the paper.
For $n,m\geq1$ and $\rho\in\mc C$ we denote  by
$$
u_{\rho}(n,m):=
L(\nu^{\tfrac{m-1}2}\delta(\rho,n),\nu^{\tfrac{m-1}2-1}\delta(\rho,n),
\dots,
\nu^{-\tfrac{m-1}2}\delta(\rho,n)),
$$
and call it a Speh representation.

\subsection{Module and comodule structures for classical groups}
In this paper we consider classical groups $\Sp (2n,F)$ and split $\SO (2n+1,F)$, $n\geq 0$. We will use  their matrix realisation from \cite{MR1356358}. Such a group of rank $n$ will be denoted by $S_n$ (we will always work with a fixed series of groups). We fix in $S_n$ a minimal parabolic subgroup consisting of all upper triangular matrices in the group. Now for each $0\leq k\leq n$, there is a standard parabolic subgroup $P_{(k)}=M_{(k)}N_{(k)}$ such that $M_{(k)}$ is naturally isomorphic to the direct product $\GL (k,F)\times S_{n-k}$. For  representations $\pi$ and $\sigma$ of $\GL (k,F)$ and $ S_{n-k}$ respectively, one  defines
$
\pi\rtimes\sigma:=\text{Ind}_{P_{(k)}}^{S_n}(\pi\otimes\sigma).
$
Denote
$
R(S):=\underset{n\geq0}{\oplus}\mathfrak R(S_n).
$
Then $\rtimes$ lifts in a natural way to $\rtimes:R\times R(S)\rightarrow R(S)$, and in this way $R(S)$ becomes an $R$-module.
The normalised Jacquet module with respect to $P_{(k)}$ of a representation $\pi$ of $S_n$ is denoted by
$
s_{(k)}(\pi).
$
Let $\pi$  be of finite length. Then we set
$$
\mu^*(\pi):=\sum_{k=0}^n \text{s.s.}(s_{(k)}(\pi))\in R\otimes R(S),
$$
and extend $\mu^*$ additively to
$
\mu^*:R(S)\rightarrow R\otimes R(S).
$
In this way, $R(S)$ becomes an $R$-comodule.

\subsection{Twisted Hopf algebra}
Denote by $\kappa:R\otimes R\rightarrow R\otimes R$ the transposition map 
$\sum_i x_i\otimes y_i\mapsto \sum_i y_i\otimes x_i$, and by
\begin{equation}
\label{twisted hopf}
M^*:=(m\otimes \text{id}_R)\circ (\sim\otimes m^*)\circ \kappa\circ m^*:
R\rightarrow R\otimes R.
\end{equation}
Then for finite length  representations $\pi$ and $\sigma$ of $\GL (n,F)$ and $ S_{k}$ respectively, by \cite{MR1356358} we have
$$
\mu^*(\pi\rtimes \sigma)=M^*(\pi)\rtimes \mu^*(\pi).
$$
Denote by $M^*_{GL}(\pi)\otimes 1$ the component of $M^*(\pi)$ which is in $\mathfrak R(\GL (n,F))\otimes \mathfrak R(\GL (0,F))$. We calculate $M^*_{GL}(\pi)$ by  the following simple formula: if $m^*(\pi)=\sum_i x_i\otimes y_i, $ then 
$
M^*_{GL}(\pi)=\sum_i x_i\times \tilde y_i.
$
The component of $M^*(\pi)$ which is in $\mathfrak R(\GL (0,F))\otimes \mathfrak R(\GL (n,F))$ is 
$
1\otimes\pi.
$

Let additionally $\sigma$ be an irreducible cuspidal representation of a classical group, and $\tau$ a subquotient of $\pi\rtimes\sigma$. Then we denote
$
s_{\GL }(\tau):= s_{(n)}(\tau).
$
Now for a  finite length representation $\pi'$ of $\GL (n,F)$ we have 
$
\ssm(s_{\GL }(\pi'\rtimes \tau))=M^*_{GL}(\pi')\times \ssm(s_{GL}(\tau)).
$

\subsection{Langlands classification for classical groups} 
\label{LC CG}
Denote by $\Irrcl$ the set of equivalence classes of all irreducible  representations of groups $S_n,n\geq0$, and by  $\Tempcl$ the subset of the tempered representations in $\Irrcl$. Let
$
\Disc_+=\{\delta\in \Disc: e(\delta)>0\}.
$
Take
$
t=((\delta_1,\delta_2,\dots,\delta_k),\tau)\in
M(\Disc_+)\times \Tempcl.
$
Chose a permutation $p$ of $\{1,\dots,k\}$ such that
$
e(\delta_{p(1)})\geq e(\delta_{p(2)})\geq\dots\geq e(\delta_{p(k)}).
$
Then, the representation
$
\lambda(t):=\delta_{p(1)}\times\delta_{p(2)}\times\dots\times\delta_{p(k)}\rtimes \tau
$
has a unique irreducible  irreducible quotient, denoted by
$
L(t).
$
The mapping
$
t\mapsto L(t)
$
defines a bijection between $M(\Disc_+)\times \Tempcl$ and $\Irrcl$, and it is the Langlands classification for classical groups (the multiplicity of $L(t)$ in $\lambda(t) $ is one).
If $t=(d;\tau)$, then $L(d;\tau)\tilde{\ }\cong L(d;\tilde\tau)$ and  $L(d;\tau)\bar{\ }\cong L(\bar d;\bar\tau)$, where $\bar\pi$ denotes the complex conjugate representation of $\pi$.

 Introducing
$
\SC_+=\{\Delta\in\SC; e( \delta (\Delta))>0\},
$
 we can define in a natural way the Langlands classification 
 $
 (a,\tau)\mapsto L(a;\tau)
 $
  using parameters in $M(\SC_+)\times \Tempcl$. We will use this classification in this paper.

\subsection{Duality} There is a natural  involution $D_G$  on the Grothendieck group of the representations
of any connected reductive $p$-adic group $G$ (\cite{MR1285969} and \cite{MR1471867}, see also \cite{MR3769724}). It
 takes any irreducible representation to an irreducible representation up to a sign.
For any irreducible representation $\pi$, let $\pi^t$ be the irreducible representation such that $D_G(\pi)=\pm\pi^t$.
We call $\pi^t$ the Aubert involution of $\pi$, or \ASS dual of $\pi$.
This involution is compatible with parabolic induction in the sense that
$
(\pi\rtimes\tau)^t=\pi^t\rtimes\tau^t
$
(on the level of Grothendieck groups).
Furthermore, regarding Jacquet modules, the mapping
$$
\pi_1\otimes\dots\pi_l\otimes\mu\mapsto \tilde\pi_1^t\otimes\dots\tilde\pi_l^t\otimes\mu^t,
$$
is a bijection from the semi-simplification of $s_\beta(\pi)$ onto the semi simplification of $s_\beta(\pi^t)$
($\beta$ is the partition which parameterises the corresponding parabolic subgroup).

We will use  the following  result: for $\D\in \SC$ and cuspidal $\sigma \in \Irrcl$  we have
$$
\delta(\D)\rtimes\s \text{ is reducible }\iff \theta\rtimes\s \text{ is reducible for some }\theta \in \D.
$$
This result follows from  Theorem 13.2. of \cite{MR1658535}.
To get the above result from  the this theorem, one needs to know that the cuspidal reducibility exponents are in $\tfrac12\Z$, which is implied by the basic assumption from \cite{MR1896238}. This assumption   follows from  (ii) in Remarks 4.5.2 of \cite{MR2305609} and Theorem 1.5.1 in \cite{MR3135650}.

An irreducible representation will be called cotempered, if it is the Aubert involution of a tempered representation.

\subsection{Some formulas for $M^*$} Let $\rho\in\mc C$ be selfcontragredient and $[x,y]^{(\rho)}\in\mc S(\mc C)$.
Then, one easily gets  
\begin{equation} \label{M-seg}
M^*\big(\delta([x,y]^{(\rho)})\big) =
\sum_{i= x-1}^{ y}\sum_{j=i}^{ y}
\delta([-i,-x]^{(\rho)})\times\delta([j+1,y]^{(\rho)}) \otimes\delta([i+1,j]^{(\rho)}).
\end{equation}
In the above formula and the formulas below, we take terms of the form $[t,t-1]^{(\rho)}$ to be the identity of $R$, i.e. to be $L(\emptyset)$. In particular
\begin{equation}\label{M-seg-GL}
M_{\GL}^*\big(\delta([x,y]^{(\rho)})\big) =
\sum_{i= x-1}^{ y}\delta([-i,-x]^{(\rho)})\times\delta([i+1,y]^{(\rho)}).
\end{equation}
We denote the multiset 
$
([x]^{(\rho)}, [x+1]^{(\rho)},\dots, [y]^{(\rho)})
=
([x]^{(\rho)})+( [x+1]^{(\rho)})+\dots+( [y]^{(\rho)})
$
by  
$$
([x,y]^{(\rho)})\tr.
$$
Then 
$
\la\nu^x\rho,\nu^{x+1}\rho,\dots,\nu^y\rho\ra
=
L(([x,y]^{(\rho)})\tr).
$
Now
\begin{multline}
\label{M*segment-Z}
M^*(L(([x,y]^{(\rho)})\tr))=
\\
\sum_{x -1\leq i \leq y}\sum_{ x -1\leq j \leq i}
L (([-y,-i -1]^{(\rho)})\tr)\times L (([x ,j]^{(\rho)})\tr)\otimes L (([j +1,i]^{(\rho)})\tr),
\end{multline}
\begin{gather}
\label{M*segment-Z-GL}
M^*_{GL}(L(([x,y]^{(\rho)})\tr))=
\sum_{i=x-1}^y
L (([-y,-i -1]^{(\rho)})\tr)\times L (([x ,i]^{(\rho)})\tr).
\end{gather}

\subsection{Some very simple irreducible square integrable and tempered  representations of classical groups} 
\label{simple SIR}
Let $\rho$ and $\sigma$ be irreducible cuspidal representations of a general linear and a classical group respectively, and suppose that $\rho$ is selfcontragredient (i.e. $\rho\cong\tilde\rho$). Then there exists a unique $\alpha_{\rho,\sigma}\in\tfrac12\Z_{\geq0}$ such that 
$$
\nu^{\alpha_{\rho,\sigma}}\rho\rtimes\sigma
$$
reduces. We denote $\alpha_{\rho,\sigma}$ simply by $\alpha$ once  we fix $\rho$ and $\sigma$.

Suppose $\alpha>0$ and $n\geq0$. Then the representation $\delta([\nu^\alpha\rho,\nu^{\alpha+n}\rho])\rtimes\sigma$ contains a unique irreducible representation, which is denoted by
$$
\delta([\nu^\alpha\rho,\nu^{\alpha+n}\rho];\sigma).
$$
This representation is  square integrable, and it is called a generalised Steinberg representation. Further, $\delta([\nu^\alpha\rho,\nu^{\alpha+n}\rho];\sigma)$ is the unique irreducible  subrepresentation of 
$\nu^{\alpha+n}\rho\rtimes \delta([\nu^\alpha\rho,\nu^{\alpha+n-1}\rho];\sigma)$. By \cite[Theorem 6.3, (viii)]{MR1600280} we have
\begin{equation}
\label{Steinberg-mu-star}
\mu^*(\delta([\nu^\alpha \rho, \nu^{\alpha + n}\rho],\sigma)) =
\sum^n_{k=-1} \delta([\nu^{\alpha + k + 1}\rho, \nu^{\alpha +
n} \rho]) \otimes \delta([\nu^\alpha \rho, \nu^{\alpha + k}
\rho], \sigma),
\end{equation}
where we take $\delta(\emptyset,\sigma)=\sigma$.

Starting from generalised Steinberg representations, one can construct further (strongly positive) square integrable representations (\cite{MR1913095} and \cite{MR1896238} contain a  general construction of such representations; see also section 34 of \cite{MR2908042}). We will   describe here only the first step of the construction. Let $\alpha\geq\tfrac32$. Take $m\in\Z_{\geq0}$ such that $m\leq n$. Then the representation 
$
\delta([\nu^{\alpha-1}\rho,\nu^{\alpha-1+m}\rho])\rtimes \delta([\nu^\alpha\rho,\nu^{\alpha+n}\rho];\sigma)
$
has a unique irreducible subrepresentation, denoted by 
$$
\delta_{\spsi}([\nu^{\alpha-1}\rho,\nu^{\alpha-1+m}\rho],[\nu^\alpha\rho,\nu^{\alpha+n}\rho];\sigma).
$$
 This representation is square integrable.

Suppose  $\alpha>0$. Take $x,y\in \R$ such that  $ x\leq y$ and 
$x-\alpha, \ y-\alpha\in \Z_{\geq0}$. Then the representation $\delta([\nu^{-x}\rho,\nu^{y}\rho])\rtimes\sigma$ contains precisely two irreducible subrepresentations. If $\alpha\in \Z$ (resp. $\tfrac12+\Z$), then precisely one of these subrepresentations contains in its Jacquet module the term $\delta([\nu\rho,\nu^{x}\rho])\times \delta([\rho,\nu^{y}\rho])\otimes\sigma$ (resp. $\delta([\nu^{\frac12}\rho,\nu^{x}\rho])\times \delta([\nu^{\frac12}\rho,\nu^{y}\rho])\otimes\sigma$). We denote this irreducible subrepresentation by
$$
\delta([\nu^{-x}\rho,\nu^{y}\rho]_+;\sigma)
$$
and the other irreducible subrepresentation by
$
\delta([\nu^{-x}\rho,\nu^{y}\rho]_-;\sigma).
$
Both subrepresentations are square integrable if $x<y$, and tempered (but not square integrable) otherwise (see \cite{MR1698200} for more details).

Assume now $\alpha=0$ and $n\geq0$. Take irreducible tempered representations 
$\delta([\rho]_\pm;\sigma)$
such that
\begin{equation}
\label{pm}
\rho\rtimes\sigma:= \delta([\rho]_+;\sigma)\oplus \delta([\rho]_-;\sigma).
\end{equation}
If $\sigma$ is generic, then we take $\delta([\rho]_+;\sigma)$ to be a generic summand.
 Then the representation $\delta([\nu\rho,\nu^{\alpha+n}\rho])\rtimes \delta(\rho_\pm;\sigma)$ contains a unique irreducible subrepresentation, which is denoted by
$
\delta([\rho,\nu^{n}\rho]_\pm;\sigma).
$
These representations are  square integrable for $n\geq1$. Further for $n\geq1$, $\delta([\rho,\nu^{\alpha+n}\rho]_\pm;\sigma)$ is the unique irreducible  subrepresentation of 
$\nu^{\alpha+n}\rho\rtimes \delta([\rho,\nu^{\alpha+n-1}\rho]_\pm;\sigma)$.

Let now
$\alpha=1$ and $n\geq1$. Then 
$
\rho\rtimes \delta([\nu\rho,\nu^{n}\rho];\sigma)
$
decomposes into a direct sum of two irreducible (tempered) representations, which we denote by
$$
\tau([\rho]_\pm; \delta([\nu\rho,\nu^{n}\rho];\sigma)).
$$
The representation $\tau([\rho]_+; \delta([\nu\rho,\nu^{n}\rho];\sigma))$ is characterised by the fact that 
$\delta([\rho,\nu^{n}\rho])\otimes\sigma$ is in its Jacquet module.

\subsection{Jantzen lemma}
\label{jantzen lemma}
 \begin{definition}  Let $\pi$ be an irreducible representation of some $S_n$ and $\rho\in\mc C$. 
\begin{enumerate}
 \item
 We let 
 $
 \mu^*_{\rho}(\pi)
 $
  be the sum (in the corresponding Grothendieck group)  of all irreducible terms in $\mu^*(\pi)$ of the form 
$\rho\otimes\tau$.

   \item
   We let 
 $$
 \Jac_{\rho}(\pi)
 $$
  be the sum in $R(S)$  of all $\tau$ when irreducible $\rho\otimes\tau$ runs over $ \mu^*_{\rho}(\pi)$. 
   \end{enumerate}
\end{definition}

\noindent
Observe that
  $
   \mu^*_{\rho}(\pi)=\rho\otimes  \Jac_{\rho}(\pi).
  $
By Lemma 5.6 of \cite{MR3713922} we have
$
 \Jac_{\rho_1}\circ  \Jac_{\rho_2}= \Jac_{\rho_2}\circ  \Jac_{\rho_1}\text{ if } \rho_1\not\in\{\nu\rho_2,\nu^{-1}\rho_2\}.
$
Below we recall of Lemma 3.1.3 from \cite{MR3268853} (in a slightly less general form).

\begin{definition} 
\label{Janzen def} 
Let $\pi$ be an irreducible representation of some $S_n$ and $\rho\in\mc C$. Denote by 
$
f_\pi(\rho)
$
  the largest value of f such that some Jacquet module of $\pi$  contains   an irreducible  subquotient  of the form 
$\rho\otimes\dots\otimes\rho\otimes\tau$, where $\rho$ shows up $f$ times.
 We let 
 $$
 \mu^*_{\{\rho\}}(\pi)
 $$
  be the sum of all irreducible terms   in $\mu^*(\pi)$ of the form 
$\rho\times\dots\times\rho\otimes\tau$, where $\rho$ shows up $f_\pi(\rho)$ times in the last formula and
   $\tau$ is irreducible.
\end{definition}

\begin{lemma}
\label{lemma-Jantzen} Let $\pi$ be an irreducible representation of some $S_n$ and $\rho\in\mc C$. Then there is a unique irreducible representation  $\theta$  and unique $f\geq0$ such that the following are all satisfied:
\begin{enumerate}
\item $\pi\h\lambda\rtimes\theta$, where $\lambda:=\rho\times\dots\times\rho$ and $\rho$ shows up $f$ times in the last formula.
\item $\mu^*_{\rho}(\theta) = 0$.
\end{enumerate}
Furthermore, $f = f_\pi(\rho)$ and all irreducible subquotients of $\mu^*_{\{\rho\}}(\pi)$ are isomorphic to $\lambda\otimes\theta$.

 If   $\rho\not\cong\tilde \rho$, then $\mu^*_{\{\rho\}}(\pi)$ is irreducible (i.e. the multiplicity of $\lambda\otimes\theta$ in $\mu^*_{\{\rho\}}(\pi)$ is one) and $\pi\h\lambda\rtimes\theta$ is its unique irreducible subrepresentation. In particular, if $\pi'$ is an irreducible representation with $\mu^*_{\{\rho\}}(\pi)=\mu^*_{\{\rho\}}(\pi')$, then $\pi\cong \pi'$.
 
\end{lemma}

\begin{remark}
\label{rem jantzen}
\begin{enumerate}
\item If $ \mu^*_{\rho}(\pi)=\rho\otimes \theta$ , then $ \mu^*_{\tilde\rho}(\pi\tr)=\tilde\rho\otimes \theta\tr$.

\item $ \Jac_{\rho}(\pi)\tr= \Jac_{\tilde\rho}(\pi\tr)$.

\item
\label{for comp inv}
 If $\rho\not\cong\tilde\rho$ and $ \mu^*_{\{\rho\}}(\pi)=\lambda\otimes \theta$ , then $ \mu^*_{\{\tilde\rho\}}(\pi\tr)=\tilde\lambda\otimes \theta\tr$ and
$
\pi\tr\h\tilde\lambda\rtimes \theta\tr
$
as the unique irreducible subrepresentation.

\item 
Let $\rho\in\mc C$, $x\in\R$, $x{\ne0}$, and let $\pi,\pi'$ be irreducible representations of $S_n$ and $S_m$. Suppose $\pi\h\nu^x\rho\rtimes\pi'$, $\Jac_{\nu^{-x}\rho}(\pi)\ne0$ and $\Jac_{\nu^{-x}\rho}(\pi')=0.$ Then $\nu^x\rho\rtimes\pi'$ is irreducible $($\cite[Remark in 2.3]{MR2209850}$)$.
\end{enumerate}
\end{remark}

\section{Parameters of A-packets}
\label{MPAP}

In this section we recall the well-known terminology related to A-packets following mainly C. M\oe glin.

\subsection{A-parameters} For an irreducible cuspidal representation $\rho$ of $GL(n_\rho,F)$ (this defines $n_\rho$),   $\rho$ will also  denote the corresponding irreducible representation of the Weil group $W_F$ under the local Langlands correspondence for general linear groups. The  irreducible algebraic representation of $\slc$ of dimension $a$ over $\C$ is denoted by 
$
E_a.
$
A triple $\rab$, $\rho\in \mc C$, $a,b\in\Z_{>0}$ is called a Jordan block.
To shorten notation in the paper, we  denote
$$
E_{a,b}^\rho:=\rho\otimes E_a\otimes E_b.
$$
For a connected reductive group $G$ over $F$, the connected component  of the dual group $^LG$ is denoted by $^LG^0$, and called the complex dual group. Then $^L \hskip-2pt \Sp(2n,F)^0=\SO(2n+1,\C)$ and $^L \hskip-2pt \SO(2n+1,F)^0=\Sp(2n,\C)$. Set $n^*=2n+1$ (resp. $n^*=2n$) if $G=\Sp(2n,F)$ (resp. $G=\SO(2n+1,F)$).

\begin{definition}An A-parameter for the group $S_n$ is  a continuous homomorphism $\psi: \wddg \rightarrow \ ^L \hskip-2pt S_n^0 $, which is bounded on $W_F$ and is (complex) algebraic on $\slc\times\slc$.
\end{definition}

We can decompose  $\psi$ as above into a the sum of irreducible representations
\begin{equation}
\label{psi-jord}
\psi=
\underset{\rab\in\Jord(\psi)}
{\boldsymbol\oplus} \ E_{a,b}^\rho
\end{equation}
where 
$
Jord(\psi)
$
is a finite multiset,  which is called the   Jordan block of $\psi$. Then we have $\sum_{\rab\in\Jord(\psi)}n_\rho a b=n^*$. Clearly, $\jp$ determines $\psi$ (up to an equivalence). We can work with $\Jord(\psi)$ as the A-parameter instead of $\psi$.
For a finite multiset   of  $(\rho,a,b)$'s, some additional conditions may  be needed so that this multiset is the set of Jordan blocks of an A-parameter (we will not discuss these conditions here). Denote
$$
\jrp=((a,b);(\rho,a,b)\in\Jord(\psi)).
$$
The set of all equivalence classes of A-parameters of $S_n$ will be denoted by $\Psi(S_n)$, and $\Psi=\cup_{n\geq 0}\Psi(S_n)$ (we will indicate the series of  groups we are woring  with  when this is necessary).

One  says that $\rab\in Jord(\psi)$ has good parity (with respect to $S_n$) if there exist $z\in \Z$ 
such that
$
\nu^{\frac{a+b}2+z}\rho\rtimes 1_{S_0}
$
reduces. We say that $ \Jord(\psi)$ has good parity if each of  its elements has good parity.
The subset of A-parameters of good parity will be denoted by
$$
\Pgp.
$$ 
In this paper  we will only work  with A-parameters  of good parity. This implies that the cuspidal representation $\rho$ will  always be selfcontragredient.

To an A-parameter $\psi$  one can associate an irreducible unitarizable representation
$$
\pi_\psi:=\underset{(\rho,a,b)\in Jord(\psi)} \times u_\rho(a,b)
$$
of a general linear group over $F$ (we can work with $\pi_\psi$ as the A-parameter instead of $\psi$).

\subsection{Another parameterisation  of Jordan blocks} 
\label{another-par}
Let $(\rho,a,b)\in Jord(\psi)$, $\psi\in\Pgp$. Put
 $$ 
 A=\tfrac{a+b}2-1, \quad B=\tfrac{|a-b|}2
 $$
and  $\zeta_{a,b}=sign(a-b)$ if $a\ne b$, and $\zeta_{a,b}=1$ arbitrary element of $\{\pm1\}$ otherwise.
 Obviously  either $A,B\in \Z_{\geq0}$ or $A,B\in \tfrac12+\Z_{\geq0}$.
 Observe that
 $$
 a=A+1+\zeta_{a,b}B, \qquad b  =A+1-\zeta_{a,b}B.
 $$
The Jordan block $(\rho,a,b)$ will be also denoted  by 
$$
(\rho,A,B,\zeta_{a,b}).
$$

\subsection{Modifying  A-parameters}
\label{modifying}
Fix a series of classical groups   and  $\psi\in\Pgp$. 
Let $\rab$ and $\rabp$ be two Jordan blocks which satisfy
\begin{equation}
\label{both parity}
\tilde\rho\cong\rho, \qquad a\equiv a' (\text{mod }2\Z), \qquad b\equiv b' (\text{mod }2\Z).
\end{equation}
Then:

\begin{enumerate}
\item
If $\rab\in\jrp$, and if we define a new parameter $\psi'$ by replacing $\rab$ by $\rabp$ in $\Jord(\psi)$, then
$
\psi'\in\Pgp.
$

\item
If $\rab$ has good parity, then
 $
 \psi \ \oplus \    E_{a,b}^\rho \ \oplus \ E_{a',b'}^\rho
 $
 has good parity as well.
 
 \item
If $\rab$ has good parity and $ab\in2\Z$, then 
$
 \psi \ \oplus \    E_{a,b}^\rho
$
 has good parity as well.
\end{enumerate}

\subsection{Elementary A-parameters} 
A Jordan block $\rab$ will be called elementary if $1\in\{a,b\}$.
An A-parameter $\psi$  will be called elementary if it has good parity and if  each $(\rho,a,b)\in Jord(\psi)$ is elementary.
The last condition means  that for each $(\rho,A,B,\zeta)\in Jord(\psi)$, we have $A=B$. The subset of elementary A-parameters in $\Psi$ (and $\Pgp$) is denoted by
$$
\Pe.
$$
 Let $\psi$ be elementary. Then using the parameterisation introduced in \ref{another-par}, each element in the Jordan block can be written as 
 $
 \textstyle
 (\rho, \frac{c-1}2,\frac{c-1}2,\delta_c),
 $
  where $c\in \Z_{> 0}$,  $\delta_c\in\{-1,1\}$, and we denote this  Jordan block simply  by
  $$
 (\rho,c,\delta_c)
  $$
  In the case of elementary A-parameters, we take
  $
  \delta_1=1
  $.
   Observe that 
   if $\delta_c=1$ (resp. $\delta_c=-1$), the corresponding Speh representation is  
   square integrable
(resp. the Aubert dual of a square integrable representation).

 \begin{definition}  We say that an A-parameter $\psi$ is tempered (resp. cotempered) if $b=1$ (resp. $a=1$) for each $\rab\in \jp$.
 \end{definition}

\subsection{Discrete  A-parameters}
Denote by $\Phi(S_n)$ the set of equivalence classes of admissible homomorphisms $W_F\times\slc\rightarrow ^L \hskip-6pt S_n^{\hskip1pt 0}$, and let $\Phi=\cup_{i\geq0} \Phi(S_n)$. Let $\Phi_2$ be the subset corresponding to the square integrable $L$-packets. For $\phi\in \Phi$, one defines $\Jord(\phi)$ and $\Jord_\rho(\phi)$ analogously as in the case of A-packets.

Denote by $\D: \slc \rightarrow \slc\times\slc $ the diagonal map. Let $\psi$ be an A-parameter. Then  the composition $\psi\circ \D$ is given by  $(w,g) \mapsto \psi(w,g,g)$, and this element of $\Phi$  is denoted by
$
\psi_d.
$
Then
\begin{equation}
\label{decom-diag}
\psi_d=\underset{\rab\in \jp}{\boldsymbol\oplus}\ \ \underset{j\in [B,A]}{\boldsymbol\oplus} \ \ \rho\otimes E_{2j+1},
\end{equation}
where $B$ and $A$ are defined in \ref{another-par} (i.e. $B=\frac{|a-b|}2$ and $A=\frac{a+b}2-1$).
One says that an A-parameter $\psi$ is discrete (or that has discrete diagonal restriction) if
 $
 \psi_{\text d}\in\Phi_2.
 $
It is equivalent to the fact that $\psi$ has good parity and that $\psi_d$ is a multiplicity one representation (in particular, then $\psi$ is a multiplicity one representation). 
The subset of all $\psi\in \Psi$ which are discrete  is denoted by
$$
\Pddr.
$$
By \ref{decom-diag}, an A-parameter $\psi\in \Psi$ is in $\Pddr$ 
 if and only if $\psi$ has good parity  and for each fixed $\rho$, all the segments $[B,A]$ when $(\rho,A,B,\zeta)$ runs over $\Jord(\psi)$, are disjoint.

\subsection{Characters of the component group - good parity case} Let $\psi\in\Pgp$. Then the characters of the component group can be identified with the functions $\e$ on the multiset $\Jord(\psi)$ into $\{\pm1\}$ which satisfy 
$\prod_{\rab\in\jp}\e\rab=1$
and
$$
\e\rab=\e(\rho',a',b') \text{ whenever } \rab=(\rho',a',b').
$$
Sometimes we will look at characters of the component group of $\psi$ as functions on irreducible constituents of $\psi$ in a natural way.

\begin{definition}
\label{deforming}
Let $\psi\in\Pgp$, let $\e$ be a character of the component group of $\psi$, $\rab\in\jrp$ and let $\rabp$ be a Jordan block  such that $\rabp\not\in\jrp$, $a\equiv a' (\text{mod }2\Z)$ and  $b\equiv b' (\text{mod }2\Z)$. Denote by $\psi'$ the A-parameter obtained from $\psi$ by replacing $\rab$ by $\rabp$ in $\jrp$ (then $\psi'\in\Pgp$). Denote by $\e'$ the character of the component group of $\psi'$ such that  $\e'\rabp=\e\rab$, and that $\e'$ and $\e$ coincide on the remaining elements. We say that $(\psi',\e')$ is obtained from $(\psi,\e)$ deforming $\rab$ to $\rabp$ (or deforming $E_{a,b}^\rho$ to $E_{a',b'}^\rho$).
\end{definition}

\subsection{A-packets}
\label{AP}

To each A-parameter $\psi$ of $S_n$, J. Arthur has attached in \cite[Theorem 2.2.1]{MR3135650}  a finite multiset $\Pi_\psi$ of irreducible unitarizable representations of $S_n$, called the  A-packet of $\psi$, 
 such that endoscopic distribution properties are satisfied.  We will not recall these properties, but we will follow C. M\oe glin explicit representation-theoretic construction of 
A-packets (\cite{MR2209850}, \cite{MR2533005} and \cite{MR2767522}). Unlike $L$-packets,  A-packets do not need to be disjoint for different conjugacy classes of $\psi$ (see Corollary 4.2 of \cite{MR2305609} for more information in that direction).
Another difference  from
$L$-packets is that  A-packets always consist  only  of   unitarizable representations.

More precisely, Arthur has attached to each character $\e$ of the component group of $\psi$ a multiset $\pi(\psi,\e)$ of irreducible representations.
Their sum is $\Pi_\psi$.
M\oe glin has proved that $\pi(\psi,\e)$ is multiplicity one (\cite{MR2767522}), and that for a fixed $\psi$, the  $\pi(\psi,\e)$'s are disjoint for different $\e$'s. Therefore, she has proved that 
 $\Pi_\psi$'s also have  multiplicity one. M\oe glin and Arthur definitions of $\pi(\psi,\e)$ are not the same, but they are simply related (see \cite{MR3679701}). In this paper we will  follow the M\oe glin definition of $\pi(\psi,\e)$.

M\oe glin has also proved that in the case of elementary discrete A-parameters, $\pi(\psi,\e)$ are always irreducible representations (\cite{MR2209850}).
Note that in the case of elementary discrete A-parameter $\psi$, the number of characters of the component group of $\psi$ is the same as the number of characters of the component group of $\psi_d$, and we can identify them in an obvious way.

\begin{remark}
For  $\psi_0\in\Psi$ and  $\psi_1=\rho\otimes E_a\otimes E_b$ denote $\psi=\psi_1\oplus\psi_o\oplus\tilde\psi_1$. Then there exists a canonical injection $\Pi_{\psi_0}\h\Pi_\psi$ and all irreducible constituents of $u_\rho(a,b)\rtimes\pi_0$, $\pi_0\in\Pi_{\psi_0}$, are contained in the image of this injection (\cite[section 5]{AtMiZel}; there is a more precise statement regarding $u_\rho(a,b)\rtimes\pi_0$).

\end{remark}

 \subsection{Notation \texorpdfstring{$b_{\rho,\psi,\e}$ and $a_{\rho,\psi,\e}$}{Lg}}
 \label{def b a}
  Fix $\psi\in\Pe\cap \Pgp$, a character $\e$ of the component group of $\psi$ and  selfcontragredient $\rho\in\mc C$.
Let $X$ be a subset of $\Jord(\psi)$  
of the form 
  $ X=\{(\rho,c_1,\delta_{c_1}),\dots ,(\rho,c_k,\delta_{c_k})\}$, and chose an enumeration such that
   $c_1<c_2<\dots<c_k$. We  say that $\e$ is cuspidal on $X$ if
   \begin{enumerate}
   \item $c_1\in\{1,2\}.$
   
   \item $c_{i+1}-c_i=2$ for $1\leq i\leq k-1$.
   
   \item $\e(\rho,c_{i+1},\delta_{c_{i+1}})=-\e(\rho,c_{i},\delta_{c_{i}})$ for $1\leq i\leq k-1$.
   
   \item $\e(\rho,2,\delta_2)=-1$ if $c_1=2$.
   
   \end{enumerate}
   Denote by 
$$
b_{\rho,\psi,\e}
$$
 the maximal positive integer (if it exists) such that  $\e$  is cuspidal on
 $
 \{(\rho,c,\delta_c)\in\Jord(\psi);c\leq b_{\rho,\psi,\e}\}.
 $
 If there is no integer as above, we take $b_{\rho,\psi,\e}=-1$ if elements of $\Jord_\rho(\psi_d)$ are odd, and $b_{\rho,\psi,\e}=0$ if elements of $\Jord_\rho(\psi_d)$ are even.
 Further, let 
 $$
 a_{\rho,\psi,\e}
 $$
 be
 the minimum of the set 
 $
\{c;(\rho,c,\delta_c)\in\Jord(\psi),  c> b_{\rho,\psi,\e}\}
 $
if the above set is non-empty. Otherwise, put $a_{\rho,\psi,\e}=\infty$. Note that $a_{\rho,\psi,\e}\geq3$ if $\jrpd\subseteq1+2\Z$. Since  $a_{\rho,\psi,\e}\geq b_{\rho,\psi,\e}+2$, we have the following definition:

\begin{definition}
 If $a_{\rho,\psi,\e}= b_{\rho,\psi,\e}+2$, then we say that we are in the boundary case.
\end{definition}

We  use  M\oe glin construction of A-packets in the paper, but we  do not recall it here. We  recall only  the following simple step which we  use most often:

\subsection{Simple reduction step: case of \texorpdfstring{$a_{\rho,\psi,\e}>b_{\rho,\psi,\e}+2$ or $b_{\rho,\psi,\e}=0$}{Lg}  {\rm (\cite[section 2.4, 1]{MR2209850}  or  \cite[Definition 6.3, (2)]{MR3679701})}} 
\label{simple rs} \
 We consider two possibilities.

If $a_{\rho,\psi,\e}>2$,  then the pair $(\psi',\e')$ is is obtained from $(\psi,\e)$ deforming       
$(\rho, a_{\rho,\psi,\e},\delta _{a_{_{\rho,\psi,\e}}})$ to $(\rho, a_{\rho,\psi,\e}-2,\delta _{a_{_{\rho,\psi,\e}}})$.
If $a_{\rho,\psi,\e}=2$,  then the pair $(\psi',\e')$ is defined by deleting $(\rho,a_{\rho,\psi,\e},\delta _{a_{_{\rho,\psi,\e}}})$ from $\jrp$, 
 and taking $\e'$ to be the restriction of $\e$.

\begin{definition}
If $a_{\rho,\psi,\e}>b_{\rho,\psi,\e}+2$ or $b_{\rho,\psi,\e}=0$, one defines
$$
\pi(\psi,\e)\hui \nu^{\delta _{a_{\rho,\psi,\e}}\frac{a_{\rho,\psi,\e}-1}2}\rho\rtimes \pi(\psi',\e')
$$
to be  the unique irreducible   subrepresentation of the right-hand side.
\end{definition}

\subsection{Irreducible square integrable representations}
These representations of $S_n$ decompose into a disjoint union
$
\bigsqcup \ \Pi_\Psi
$
when $\psi$ runs over all tempered discrete A-parameters of $S_n$.
For any character $\psi$ of the component group, $\pi(\psi,\e)$ is an irreducible representation.
In this situation one usually works with the Weil-Deligne group (and drops the $b$'s which are always one in this case). Then we are in the case of local Langlands correspondence for square integrable representations.

The classification  (modulo cuspidal data) of irreducible square integral representations of the groups $S_n$ is completed in  \cite{MR1896238}. To an  irreducible square integral representation  one attaches an admissible triple  $(\Jord(\pi),\e_\pi,\pi_{cusp})$ consisting of Jordan blocks, a partially defined function and a partial cuspidal support of $\pi$. Such triples classify irreducible square integrable  representations (see \cite{MR1896238} for details). Then $\Jord(\pi)=\jp$ if and only if $\pi\in\Pi_\psi$ (\cite[Theorem 1.3.1]{MR2767522} or Theorem 10.1 of \cite{MR3713922}). Further, $\e_\pi$ is the restriction of the character of the component group of $\psi$ which is attached to $\pi$ by Arthur (Theorem 10.1 of \cite{MR3713922}, see also Proposition 8.1 there ).

\subsection{A consequence of the involution}
\label{con of inv}
 Let $(\psi,\e)$ be a pair consisting of $\psi\in\Psi $ and   a character $\e$ of the component group of $\psi$. One defines a pair  
$$
(\psi\tr,\e\tr)
$$
 where $\psi\tr\in\Psi $ and  $\e\tr$ is a character  of the component group of $\psi\tr$, by the requirement that $\Jord(\psi\tr)$
consists of all $(\rho,b,a)$, $(\rho,a,b)$ in $\Jord(\psi)$, and $\e\tr$ is defined using natural bijection between $\Jord(\psi\tr)$ and $\Jord(\psi)$.

Let  $\psi\in\Pe\cap \Pddr$.
 Obviously $\psi_d=(\psi\tr)_d$. Using this, we  identify characters of component groups of $\psi$ and $\psi_d$. Therefore, if $\psi',\psi''\in \Pe\cap\Pddr$ such that $(\psi')_d=(\psi'')_d$, their characters of component groups can be identified in a natural way. 

C. M\oe glin 
defined in \cite{MR2209850}  involutions on irreducible representations, which generalise the Aubert involution, and showed that each element   $\ppe$ of an elementary discrete A-packet  can be obtained from square integrable representation corresponding to $\e$ in the $L$-packet of $\psi_d$ by applying the involution (see \cite[Theorem 5]{MR2209850}). 
A consequence of it for classical Aubert  involution  is that
\begin{equation}
\label{inv-formula-el}
\ppe\tr=\pi(\psi\tr,\e)
\quad \text{ for }\quad \psi\in\Pe\cap\Pddr
\end{equation}
 (\cite[Theorem 5]{MR2209850}, Theorem 6.10 of \cite{MR3679701}).

\subsection{Cuspidal representations in elementary discrete A-packets}
 Let $\psi\in\Pe\cap\Pddr$.  From our observations in \ref{AP} (or in \ref{con of inv}) it
follows that the  cardinality of the $L$-packet of $\psi_d$ is equal to the cardinality of the A-packet of $\psi$.

Partial Aubert involutions (defined in  section 4 of \cite{MR2209850}) carry irreducible non-cuspidal representations to non-cuspidal ones, and they cannot carry non-cuspidal to  cuspidal ones. This implies that for
an irreducible cuspidal representation $\sigma$ of a classical group, $\sigma$ belongs to the $L$-packet of $\psi_d$ if and only if it belongs to the A-packet of $\psi\in \Pi_\psi$. Moreover, they determine the same character of the component groups (after we identify them).

\subsection{Orders  on Jordan blocks} Let $\psi\in\Psi$. Any total order $>_\psi$ on $\jrp$ which satisfies
\begin{equation}
\label{P}
\tag{$\mathcal P$}
a+b>a'+b', \quad |a-b|>|a'-b'|,\quad \zeta_{a,b}=\zeta_{a',b'}\quad \implies \rab>_\psi \rabp
\end{equation}
 for any $\rab,\rabp\in\jrp$ 
will be called an {\bf admissible order}. 

We will always fix some total order $>'$ on the set $\{\rho;\jrp\ne\emptyset\}$, and assume for each $\rab,(\rho',a',b')\in\jrp$ that if $\rho>'\rho'$, then $\rab>_\psi(\rho',a',b')$ (we do not need to assume this, but it simplifies descriptions of admissible orders).
Therefore, for describing  an admissible order on $\jrp$, it is enough
to describe it on each $\jrp$.

Suppose $\psi\in\Pddr$. Then in \eqref{P} we have $a+b>a'+b' \iff |a-b|>|a'-b'|$ (and the condition $|a-b|>|a'-b'|$ is redundant in \eqref{P} in this case). Actually, in this case we can find an admissible order $>_\psi$ satisfying
$$
\rab>_\psi\rabp \iff a+b>a'+b'.
$$
Such an order will be called {\bf natural}.

One says that $\psi_{>\hskip-4pt>}\in \Pddr$ with a natural order $>_{\psi_{>\hskip-4pt>}}$  {\bf dominates $\psi\in\Pgp$ with respect to an admissible order $>_\psi$} on $\jrp$ if $\{\rho;\jr(\psi_{>\hskip-4pt>})\ne\emptyset\}=\{\rho;\jrp\ne\emptyset\}$, and if for each $\rho$ from the last set we have an order preserving bijection
$
(a_{>\hskip-4pt>},b_{>\hskip-4pt>})\mapsto (a,b)
$
 from $\jr(\psi_{>\hskip-4pt>})$ onto $\jrp$ 
 which
satisfies
$$
A_{>\hskip-4pt>}-A=B_{>\hskip-4pt>}-B\geq0 \quad \text{ and } \quad \zeta_{a,b}=\zeta_{a_{>\hskip-4pt>},b_{>\hskip-4pt>}}.
$$
Define the function $\boldsymbol T: \jp\rightarrow \Z_{\geq0}$ by
$
\boldsymbol T\,\rab = A_{>\hskip-4pt>}-A= B_{>\hskip-4pt>}-B.
$
Observe that 
$$ 
(\rho, a_{>\hskip-4pt>},b_{>\hskip-4pt>})= \big(\,\rho, \  \ a+(1+\zeta_{(\rho,a,b)})\,\boldsymbol T\rab,\  \ b+(1-\zeta_{(\rho,a,b)})\,\boldsymbol T\rab\,\big)
$$
(the bigger of the numbers $a$ and $b$ is increased  by $2\,\boldsymbol T\rab$, and  the smaller one is unchanged; in the case $a=b$, we increase the first $a$ or the second $a$ by $2\,\boldsymbol T(\rho,a,a)$ and leave the other one unchanged, depending on whether  we took $\zeta_{a,a}$ to be $1$ or $-1$).

\subsubsection{Orders on elementary packets} 
Let $\psi\in\Pe$. 
Any total order $>$ satisfying the condition
$
a+b>a'+b'\implies (a,b)>(a',b')
$
 for any $\rho$ and any $(a,b),(a',b')\in\jrp$ will be called standard. 
Any standard order is obviously admissible.
Let $\psi\in\Pe\cap \Pddr$. Since we have fixed a total order on $\{\rho;\jrp\ne\emptyset\}$, there is only one natural order on $\jrp$ (the standard one).

Let $\psi,\psi'\in \Pe$ 
and assume $\psi'\in\Pddr$. Suppose that we have a bijection  $\varphi:\Jord(\psi')\rightarrow \Jord(\psi)$ which for any $\rho$ induces a bijection $\rabp\mapsto \rab$ from $\Jord_\rho(\psi')$ onto  $\Jord_\rho(\psi)$
 which satisfies
$$
\max(k',l')>\max(k,l)\implies \max\varphi(k',l')\geq\max\varphi(k,l).
$$
Then any such a bijection will be called {\bf standard}.

\subsection{Cuspidal representations and reducibility exponent}
\label{rho and sigma}
Fix  an irreducible cuspidal selfcontragredient  representation $\rho$ of a general linear group and  an irreducible cuspidal  representation $\sigma$ of a classical  group. The representation $\nu^{\alpha_{\rho,\sigma}}\rho\rtimes\sigma $ reduces for a unique $\alpha_{\rho,\sigma}\geq 0$ (this defines $\alpha_{\rho,\sigma}$). Further, $\alpha_{\rho,\sigma}\in\frac12\Z_{\geq0}$, and we denote  $\alpha_{\rho,\sigma}$ simply by
$$
\alpha.
$$
Denote by 
$
(\phi_\sigma,\e_\sigma)
$
a pair of an admissible homomorphism of the Weil-Deligne group and a character of the component group of $\phi_\sigma$, corresponding to $\sigma$ under the local Langlands correspondence. Let 
$$
\psi_\sigma:=\phi_\sigma\otimes E_1,
$$
 and lift $\e_\sigma$ to a character of the component group of $\psi$ in the natural way, and denote it again  by $\e_\sigma$. Then 
$
a_{\rho',\psi_\sigma,\e_\sigma}=\infty
$
for all selfcontragredient $\rho'\in\mc C$. Further:

\begin{enumerate}
\item Suppose $\alpha\geq1$. This is equivalent to
  $\Jord_\rho(\phi)\ne\emptyset$. Then
$
\alpha=\tfrac{\max(\Jord_\rho((\psi_\sigma)_d))+1}2.
$

\item
Suppose $\alpha<1$. Then $\alpha=0$ (resp. $\alpha=\tfrac12$) if and only if $\nu^{\frac12}\rho\rtimes1_{S_0}$
 is irreducible (resp. reducible), where $1_G$ denotes the trivial (one-dimensional) representation of a group $G$.
\end{enumerate}

\section{Case of reducibility \texorpdfstring{$>1$}{}}
\label{sec: >1}

In this section we assume that $\rho$, $\sigma $ and $\alpha$ are as in \ref{rho and sigma}, and we assume   that  
$
\alpha>1.
$

\subsection{Involution}

The proof of the following proposition and other claims in this paper that compute the Aubert involution, is based on the basic idea of \cite{MR3868005} (another possibility is to apply  \cite{AtMiZel}).

\begin{proposition}
\label{tr >1 inv}
 Let  $\alpha\geq \frac32$ and $m,n\in \Z_{\geq0}$. Denote
$$
\pi_{m,n}:=L([\alpha-1,\alpha+m]\tr;\delta([\alpha,\alpha+n];\sigma)).
$$
 Then
 \begin{equation}
 \label{prop-trans}
\pi_{n,m}\tr
=
\pi_{m,n}.
\end{equation}
\end{proposition}

\begin{proof}  We prove formula \eqref{prop-trans} by induction with respect to
\begin{equation}
\label{ind}
\min(m,n).
\end{equation}
We start with the basis of the induction, the case  $\min(m,n)=0$. To prove the formula in this case it is enough to prove that
$$
\pi_{0,n}\tr=\pi_{n,0}
$$
for $n\geq 0$. We prove it by a new induction.
For $n=0$, $\pi_{0,0}\tr=\pi_{0,0}$
by  \cite[Proposition 4.6, (2)]{MR-T-CR3}.
Suppose that the claim holds for some $n\geq0$.
Observe that
\begin{multline}
\label{emb-p-p-0}
\pi_{0,n+1}\h [-\alpha]\times [-(\alpha-1)]\rtimes \delta([\alpha,\alpha+n+1];\sigma)
\\
\h
 [-\alpha]\times [-(\alpha-1)]\times [\alpha+n+1]\rtimes  \delta([\alpha,\alpha+n];\sigma)
 \\
 \cong
  [\alpha+n+1]\times [-\alpha]\times [-(\alpha-1)]\rtimes  \delta([\alpha,\alpha+n];\sigma).
\end{multline}
From $\pi_{0,n}
  \h
   [-\alpha]\times [-(\alpha-1)]\rtimes  \delta([\alpha,\alpha+n];\sigma)$ we get
 \begin{equation}
 \label{emb-p-p-1}
  [\alpha+n+1]\rtimes 
  \pi_{0,n}
  \h
   [\alpha+n+1]\times [-\alpha]\times [-(\alpha-1)]\rtimes  \delta([\alpha,\alpha+n];\sigma).
   \end{equation}
    One checks directly  that the right hand side has a unique irreducible subrepresentation. This together with \eqref{emb-p-p-0} and \eqref{emb-p-p-1}  implies that
   $
   \pi_{0,n+1}\h 
   [\alpha+n+1]\rtimes 
   \pi_{0,n}.
   $
   Using this, one easily  gets
   $
   \mu^*_{\{[\alpha+n+1]\}}(\pi_{0,n+1}) =
   [\alpha+n+1]\otimes 
   \pi_{0,n}
   $
   (see Definition \ref{Janzen def} for notation).
   Now
\begin{equation}
 \label{emb-p-p-2}   \pi_{0,n+1}\tr
   \h
       [-(\alpha+n+1)]\rtimes 
   \pi_{0,n}\tr= [-(\alpha+n+1)]\rtimes 
   \pi_{n,0},
\end{equation}   and $\pi_{0,n+1}\tr$ is the unique irreducible subrepresentation of the right-hand side 
   (use \eqref{for comp inv} of Remark \ref{rem jantzen}, and the inductive assumption). Since $\pi_{n+1,0}$ is a unique irreducible subrepresentation of the right-hand side of \eqref{emb-p-p-2}, we get
  that
   $
   \pi_{0,n+1}\tr
  =
   \pi_{n+1,0}.
   $

   Now 
we  prove the inductive step (for induction with respect to \eqref{ind}).
Fix $k\geq0$, and suppose that formula \eqref{prop-trans} holds for all pairs $(m,n)$ such that $\min(m,n)=k$. Now take  a pair $(m',n')$ such that $\min(m',n')=k+1$. It is enough to prove formula \eqref{prop-trans} in the case of $n'\leq m'$. Denote $m'=m$ and $n'=n+1$. Then we need to prove 
formula \eqref{prop-trans} for the pair $(m,n+1)$, where $n<m$ and the inductive assumption implies that 
formula \eqref{prop-trans} holds for  the pair $(m,n)$.

 Observe that
\begin{multline}
\pi_{m,n+1}\h L( [-(\alpha+m),-(\alpha-1)]\tr)\rtimes \delta([\alpha,\alpha+n+1];\sigma)
\\
\h
 L( [-(\alpha+m),-(\alpha-1)]\tr)\times [\alpha+n+1]\rtimes  \delta([\alpha,\alpha+n];\sigma)
 \\
 \cong
  [\alpha+n+1]\times L( [-(\alpha+m),-(\alpha-1)]\tr)\rtimes  \delta([\alpha,\alpha+n];\sigma).
\end{multline}
Further, $\pi_{m,n}
  \h
    L( [-(\alpha+m),-(\alpha-1)]\tr)\rtimes  \delta([\alpha,\alpha+n];\sigma)$ implies  
 \begin{equation}
 \label{aux emb}
  [\alpha+n+1]\rtimes 
  \pi_{m,n}
  \h
   [\alpha+n+1]\times L( [-(\alpha+m),-(\alpha-1)]\tr)\rtimes  \delta([\alpha,\alpha+n];\sigma).
   \end{equation}
   Denote the representation on the right-hand side by $\Pi$.
   We will show that $\Pi$ has a unique irreducible subrepresentation by showing that 
   $$
   \gamma:= [\alpha+n+1]\times L( [-(\alpha+m),-(\alpha-1)]\tr)\otimes  \delta([\alpha,\alpha+n];\sigma)
   $$
    has multiplicity one in $\mu^*(\Pi)$.
   This will imply $
   \pi_{m,n+1}\h 
   [\alpha+n+1]\rtimes 
   \pi_{m,n}.
   $
   
    Recall
   \begin{equation}
   \label{mu-star-proof}
   \mu^*(\Pi)
   =
   M^*([\alpha+n+1])\times  M^*(L( [-(\alpha+m),-(\alpha-1)]\tr))\rtimes  \mu^*(\delta([\alpha,\alpha+n];\sigma)).
   \end{equation}
   Suppose that $\tau_1\otimes\tau_2$ is an irreducible subquotient of $\mu^*(\Pi)$ such that
   $$
   \supp(\tau_1)=\supp([\alpha+n+1]\times L( [-(\alpha+m),-(\alpha-1)]\tr))
   $$
   (obviously, $\gamma$ satisfies the assumption of $\tau_1\otimes\tau_2$).
   Now we will analyze when we can get $\tau_1\otimes\tau_2$ from the right hand side of \eqref{mu-star-proof}. Suppose that $\tau_1\otimes\tau_2\leq \gamma_1\times\gamma_2\rtimes\gamma_3$, 
   where  $\gamma_1  , \gamma_2 $ and $\gamma_3$ are terms from the sums of       
   $$
   M^*([\alpha+n+1]),  M^*(L( [-(\alpha+m),-(\alpha-1)]\tr))\text{ \  and  \ }  \mu^*(\delta([\alpha,\alpha+n];\sigma))
   $$
    respectively.    
  Considering the support of $\tau_1$, formula \eqref{Steinberg-mu-star} implies that $\gamma_3$ must be $1\otimes \delta([\alpha,\alpha+n];\sigma)$.
   
   Further, if $\gamma_1=[\alpha+n+1]\otimes1$, then considering the support of $\tau_1$, formula \eqref{M*segment-Z} (actually \eqref{M*segment-Z-GL} is enough) implies that $\gamma_2$ must be $ L( [-(\alpha+m),-(\alpha-1)]\tr) \otimes1$, and then $\tau_1\otimes\tau_2=\gamma_1\times\gamma_2\rtimes\gamma_3=[\alpha+n+1]\times L( [-(\alpha+m),-(\alpha-1)]\tr)\otimes  \delta([\alpha,\alpha+n];\sigma)=\gamma$. 
   
   Suppose $\gamma_1\ne [\alpha+n+1]\otimes1$ (recall $M^*([\alpha+n+1])=1\otimes [\alpha+n+1]+[\alpha+n+1]\otimes1+[-(\alpha+n+1)]\otimes1$). Write $\gamma_2=\gamma_2'\otimes \gamma_2''$. Then $[\alpha+n+1]$ in the support of $\tau_1$ must come from $\gamma_2'$, i.e. it must be in $\supp(\gamma_2')$. Now formula \eqref{M*segment-Z} implies that in this case  $[\alpha-1]$ will also be  in $\supp(\gamma_2')$, and therefore also in the support of $\tau_1$, which  contradicts our assumption on the support of $\tau_1$. Therefore, we have proved multiplicity one of $\gamma$ in $\mu^*(\Pi)$ which implies $
   \pi_{m,n+1}\h 
   [\alpha+n+1]\rtimes 
   \pi_{m,n}.
   $

The last relation and Frobenius reciprocity imply 
\begin{equation}
\label{proof-ineq}
[\alpha+n+1]\otimes 
   \pi_{m,n} \leq
   \mu^*(\pi_{m,n+1}) .
   \end{equation}
      Observe that $\pi_{m,n}\leq L([\alpha-1,\alpha+m]\tr)\rtimes\delta([\alpha,\alpha+n];\sigma))
   $ and \eqref{Steinberg-mu-star}  imply 
   $$
   s_{GL}( \pi_{m,n})\leq 
  M^*_{GL}( L([\alpha-1,\alpha+m]\tr))\times\delta([\alpha,\alpha+n])\otimes \sigma,
  $$
  which further implies that $\mu^*(\pi_{m,n})$ does not have an irreducible subquotient of the form $[\alpha+n+1]\otimes -$ (use \eqref{M*segment-Z-GL} and \eqref{Steinberg-mu-star}). Now this, \eqref{proof-ineq} and Lemma \ref{lemma-Jantzen} imply
$$
   \mu^*_{\{[\alpha+n+1]\}}(\pi_{m,n+1}) =
   [\alpha+n+1]\otimes 
   \pi_{m,n}.
   $$
  Further, \eqref{for comp inv} of Remark \ref{rem jantzen} and the inductive assumption imply
    $\pi_{m,n+1}\tr \h
   [-(\alpha+n+1)]\rtimes 
   \pi_{n,m}$, which implies $\pi_{m,n+1}\tr=\pi_{n+1,m}$ since $\pi_{n+1,m}$ embeds into $
   [-(\alpha+n+1)]\rtimes 
   \pi_{n,m}$ as the unique irreducible subrepresentation. This completes the proof of the proposition
\end{proof}

\subsection{Definition of $(\psi,\e)_{k,l}$ in the case $\alpha>1$}
\label{notation>1}
The assumption $\alpha>1$ and \ref{rho and sigma} imply that 
$
\psi_\sigma=\psi_{-}\oplus E_{ 2\alpha-3,1}^\rho\oplus E_{ 2\alpha-1,1}^\rho
$
for some $\psi_-\in\Pe\cap\Pddr$ (this defines $\psi_-$). 
Denote in the sequel
$$
\psi_{k,l}:= \psi_-\oplus E_{ k,1}^\rho\oplus E_{1, l}^\rho,
$$
where $k,l\geq 0$ will be always chosen to be of the same parity as $2\alpha-1$ (and therefore $\psi_{k,l}\in\Pgp$, which implies $\psi_{k,l}\in\Pe$). In the sequel, we will always chose $k$ and $l$ such that $\psi_{k,l}$ is a multiplicity one representation. Clearly, $(\psi_{2\alpha-1,2\alpha-3})_d=(\psi_\sigma)_d$.

Further,  denote by $\e_{k,l}$ the character of the component group of $\psi_{k,l}$ which extends  $\e_\sigma$ on  $\psi_-$, and which satisfies
\begin{gather*}
\e_{k,l}(E_{ k,1}^\rho)= \e_\sigma (E_{ 2\alpha-1,1}^\rho),
\\
\e_{k,l}(E_{ 1,l}^\rho)=\e_\sigma (E_{ 2\alpha-3,1}^\rho)\text{ if }\alpha>\tfrac32,\text{ and }
\e_{k,l}(E_{ 1,l}^\rho)=1\text{ if }\alpha=\tfrac32. 
\end{gather*}
We will work throughout  in this section with pairs $(\psi_{k,l},\e_{k,l})$. Therefore, to shorten notation in the sequel, we denote such a pair by 
$
(\psi,\e)_{k,l}.
$

\subsection{On corresponding A-packets}

\begin{theorem} 
\label{theorem >1}Let  $\alpha\geq \frac32$ and $m,n\in\Z_{\geq0}$. Using the notation for A-parameters  introduced  above,
 we have

\begin{enumerate} 
\item
\label{packet a > 3/2}
 $L([\alpha-1,\alpha+m]\tr;\delta([\alpha,\alpha+n];\sigma))\in \Pi_{\psi_{2\alpha+1+2n,2\alpha+1+2m}}$.
In particular,  $L([\alpha-1,\alpha+m]\tr;\delta([\alpha,\alpha+n];\sigma))$ is unitarizable.
\item
\label{ppe 3/2 formula} For $m\ne n$ we have
\begin{equation}
\label{formula>1}
\pi((\psi,\e)_{2\alpha+1+2n,2\alpha+1+2m})
=L([\alpha-1,\alpha+m]\tr;\delta([\alpha,\alpha+n];\sigma)).  
\end{equation}
\end{enumerate}

\end{theorem}

\begin{proof} 
The proof goes through several steps.
\subsubsection{Case of $m=-2$} 
\label{first step ref}
Using induction, we first prove the (well-known) simple fact that
$$
\pi((\psi,\e)_{2\alpha+1+2n,2\alpha-3}))=\delta([\alpha,\alpha+n];\sigma), \quad n\geq -1.
$$
 Observe first that
 $
 \sigma=
 \pi((\psi,\e)_{2\alpha-1,\,2\alpha-3}).
 $
 Therefore, we have a basis of induction. Suppose $n\geq 0$ and that the above formula holds for $n-1$.
 Consider now $(\psi,\e)_{2\alpha+1+2n,2\alpha-3}$. 
 Then
\begin{gather*} b_{\rho,(\psi,\e)_{2\alpha+1+2n,2\alpha-3}}=2\alpha-3,
\quad
 a_{\rho,(\psi,\e)_{2\alpha+1+2n,2\alpha-3}}=2\alpha+1+2n,
\\
 \delta_{a_{\rho,(\psi,\e)_{2\alpha+1+2n,2\alpha-3}}}=1. 
\end{gather*}
Now by \ref{simple rs}
we know that $\pi((\psi,\e)_{2\alpha+1+2n,2\alpha-3})$ is the  unique irreducible subrepresentation of
$$
[(\delta_{(\psi,\e)_{2\alpha+1+2n,2\alpha-3}})\tfrac{(2\alpha+1+2n)-1}2]\rtimes \pi((\psi,\e)_{2\alpha-1+2(n-1),2\alpha-3})
=[\alpha+n]\rtimes\delta([\alpha,\alpha+n-1];\sigma),
$$
which easily implies  
$
\pi((\psi,\e)_{2\alpha+1+2n,2\alpha-3})
=
\delta([\alpha,\alpha+n];\sigma),
$
and completes the proof of the inductive step.

\subsubsection{Proof of \eqref{formula>1} for $-2\leq m<n$} 
\label{-2 leq m<n}
 We have proved above that the claim holds for $m=-2$. We now fix some $m\geq-1$ (together with $n>m$), and assume that the formula \eqref{formula>1} holds for $m-1$. We will prove by induction that it holds for $m$.
 Now
\begin{gather*} b_{\rho,(\psi,\e)_{2\alpha+1+2n,2\alpha+1+2m}}=2\alpha-5,
\quad
 a_{\rho,(\psi,\e)_{2\alpha+1+2n,2\alpha+1+2m}}=2\alpha+1+2m,
\\
 \delta_{a_{\rho,(\psi,\e)_{2\alpha+1+2n,2\alpha+1+2m}}}=-1,
\end{gather*}
with the exception that for $\alpha=\frac32$ we take $b_{(\psi,\e)_{2\alpha+1+2n,2\alpha-1}}=0$.
By \ref{simple rs}
we know that $\pi((\psi,\e)_{\rho,(\psi,\e)_{2\alpha+1+2n,2\alpha+1+2m}})$ is the unique irreducible subrepresentation of
\begin{multline}
[(\delta_{\rho,(\psi,\e)_{2\alpha+1+2n,2\alpha+1+2m}})\tfrac{(2\alpha+1+2m)-1}2]\rtimes
 \pi((\psi,\e)_{(\psi,\e)_{2\alpha+1+2n,2\alpha-1+2m}})
\\
=[-(\alpha+m)]\rtimes L([\alpha-1,\alpha+m-1]\tr;\delta([\alpha,\alpha+n];\sigma)),
\end{multline}
which implies   formula \eqref{formula>1}, and completes the proof of the inductive step.

\subsubsection{Reverse setting, $n=0$}
\label{rev n=0}
 We  now repeat  the previous construction in the reversed setting.
Denote by $\e_{k,l}'$ the character of the component group of $\psi_{k,l}$ which extends  $\e_\sigma$ on  $\psi_-$, and which satisfies
\begin{gather*}
\e_{k,l}'(E_{ k,1}^\rho)= \e_\sigma (E_{ 2\alpha-3,1}^\rho)
\text{ if }\alpha>\tfrac32,\text{ and }
\e_{k,l}'(E_{ 1,l}^\rho)=1\text{ if }\alpha=\tfrac32,
\\
\e_{k,l}'(E_{ 1,l}^\rho)=\e_\sigma (E_{ 2\alpha-1,1}^\rho). 
\end{gather*}
We claim that
\begin{equation}
\label{m up only}
\pi((\psi,\e')_{2\alpha-3,2\alpha+1+2m})=L([\alpha,\alpha+m]\tr;\sigma), \quad m\geq-1.
\end{equation}
The proof goes by induction. Observe that  $\sigma=\pi((\psi,\e')_{2\alpha-3,2\alpha-1})$, which is the basis of the induction. Suppose $m\geq0$ and that the formula holds for $m-1$.
 Then
\begin{gather*} b_{\rho,(\psi,\e')_{2\alpha-3,2\alpha+1}}=2\alpha-3,
\quad
 a_{\rho,(\psi,\e')_{\rho,2\alpha-3,2\alpha+1+2m}}=2\alpha+1+2m,
\\
 \delta_{a_{\rho,(\psi,\e')_{2\alpha-3,2\alpha+1+2m}}}=-1. 
\end{gather*}
By \ref{simple rs} we know that $\pi((\psi,\e')_{2\alpha-3,2\alpha+1+2m})$ is the unique irreducible subrepresentation of
$$
[(\delta_{(\psi,\e')_{2\alpha-3,2\alpha+1+2m}})\tfrac{(2\alpha+1+2m)-1}2]\rtimes \pi((\psi,\e')_{2\alpha-3,2\alpha-1+2(m-1)})
=[-\alpha-m]\rtimes L([\alpha,\alpha+m-1]\tr;\sigma),
$$
which directly implies  formula \ref{m up only}. This completes the proof of the inductive step.

\subsubsection{Reverse setting, $-1\leq n<m$} 
\label{-1 leq n<m} 
With $\e'_{k,l}$ introduced in \ref{rev n=0},
 we now prove  the formula
\begin{equation}
\label{additional formula}
\pi((\psi,\e')_{2\alpha+1+2n,2\alpha+1+2m})=L([\alpha-1,\alpha+m]\tr;\delta([\alpha,\alpha+n];\sigma)), \ \ m>n\geq -1
\end{equation}
 by induction with respect to $n$. 
For the basis of induction for $n=-1$,  we need to consider  $(\psi,\e')_{2\alpha-1,2\alpha+1+2m}$.
Then
\begin{gather*} b_{\rho,(\psi,\e')_{2\alpha-1,2\alpha+1+2m}}=2\alpha-5,
\quad
 a_{\rho,(\psi,\e')_{2\alpha-1,2\alpha+1+2m}}=2\alpha-1,
\\
 \delta_{a_{\rho(\psi,\e')_{2\alpha-1,2\alpha+1+2m}}}=1,
\end{gather*}
with the exception that for $\alpha=\frac32$ we take $b_{(\psi,\e')_{2\alpha-1,2\alpha+1+2m}}=0$.
By \ref{simple rs} we know that $\pi((\psi,\e')_{2\alpha-1,2\alpha+1+2m})$ is the unique irreducible subrepresentation of
$$
[(\delta_{(\psi,\e')_{2\alpha-1,2\alpha+1+2m}})\tfrac{(2\alpha-1)-1}2]\rtimes \pi((\psi,\e')_{2\alpha-3,2\alpha+1+2m})
=[\alpha-1]\rtimes L([\alpha,\alpha+m]\tr;\sigma)
$$
(the above equality    follows from \ref{rev n=0}).
Therefore
\begin{multline*}
\pi((\psi,\e')_{2\alpha-1,2\alpha+1+2m})\h [\alpha-1]\times L([-(\alpha+m),-\alpha]\tr)\rtimes\sigma
\cong  L([-(\alpha+m),-\alpha]\tr)\times [\alpha-1]\rtimes\sigma
\\
\cong   L([-(\alpha+m),-\alpha]\tr)\times [-(\alpha-1)]\rtimes\sigma.
\end{multline*}
This obviously implies \eqref{additional formula} for $n=-1$.

We go now to  the inductive step. Suppose $n\geq 0$ and that  formula \eqref{additional formula} holds for $n-1$. Here
\begin{gather*} b_{\rho,(\psi,\e')_{2\alpha+1+2n,2\alpha+1+2m}}=2\alpha-5,
\quad
 a_{\rho,(\psi,\e')_{2\alpha+1+2n,2\alpha+1+2m}}=2\alpha+1+2n,
\\
 \delta_{a_{\rho,(\psi,\e')_{2\alpha+1+2n,2\alpha+1+2m}}}=1.
\end{gather*}
Then $\pi((\psi,\e')_{2\alpha+1+2n,2\alpha+1+2m})$ is the unique irreducible subrepresentation of
\begin{multline*}
[(\delta_{\rho,(\psi,\e')_{2\alpha+1+2n,2\alpha+1+2m}})\tfrac{(2\alpha+1+2n)-1}2]\rtimes \pi((\psi,\e')_{2\alpha+1+2(n-1),2\alpha+1+2m})
\\
=[\alpha+n]\rtimes L( [\alpha-1,\alpha+m]\tr ;\delta([\alpha,\alpha+n-1];\sigma)).
\end{multline*}
The last representation embeds into
\begin{multline}
\Gamma:=[\alpha+n]\times L( [-(\alpha+m),-(\alpha-1)]\tr)\rtimes\delta([\alpha,\alpha+n-1];\sigma)
\\
\cong  L( [-(\alpha+m),-(\alpha-1)]\tr)\times  [\alpha+n]\rtimes\delta([\alpha,\alpha+n-1];\sigma)
\end{multline}
We will show that the multiplicity of 
$$
\gamma:=[\alpha+n]\times L( [-(\alpha+m),-(\alpha-1)]\tr)\otimes\delta([\alpha,\alpha+n-1];\sigma)
$$
in $\mu^*(\Gamma)$ is one. To get $\gamma$ as a subquotient, from  formula
$$
\mu^*(\Gamma)=M^*( [\alpha+n])\times M^*(L( [-(\alpha+m),-(\alpha-1)]\tr))
\rtimes
\mu^*(\delta([\alpha,\alpha+n-1];\sigma)
)
$$
we see that on the last term on the right hand side we must take $1\otimes \delta([\alpha,\alpha+n-1];\sigma)$ (to see this, one can consider cuspidal supports). If we did not take $[\alpha+n]\otimes1$ from $M^*( [\alpha+n])$,  formula 
\eqref{M*segment-Z-GL} implies that we would have positive exponents different from $\alpha+n$ on the left hand side of $\otimes$. Therefore, we must take $[\alpha+n]\otimes1$, which further implies that from $M^*(L( [-(\alpha+m),-(\alpha-1)]\tr))$ we must take $L( [-(\alpha+m),-(\alpha-1)]\tr)\otimes1$. This implies  multiplicity one.

Therefore $\Gamma$ has a unique irreducible subrepresentation. Since $L( [-(\alpha+m),-(\alpha-1)]\tr)\rtimes\delta([\alpha,\alpha+n];\sigma)\h \Gamma$, 
we get that \eqref{additional formula} holds. This completes the proof of the inductive step.

\subsubsection{Case $m=n\geq0$} 
\label{n=m,>1}
Denote
$
\psi_{>\hskip-4pt>}:=\psi_{2\alpha+3+2m,2\alpha+1+2m}, \  \psi:=\psi_{2\alpha+1+2m,2\alpha+1+2m}.
$
We fix 
any standard order on $\Jord_\rho(\psi_{>\hskip-4pt>})$, and denote it by $>_{\psi_{>\hskip-4pt>}}$. 
Then this is a 
 natural order.
 Define a bijection $\Jord_\rho(\psi_{>\hskip-4pt>})\rightarrow \Jord_\rho(\psi)$ which carries
$$
\varphi:(\rho,2\alpha+3+2m,1)\mapsto (\rho,2\alpha+1+2m,1),
$$
and is equal to the identity on the remaining elements. 
Using the bijection $\varphi$, we define total order $>_{\psi}$ on $\Jord_\rho(\psi)$ (i.e. $\varphi(u)>_{\psi}\varphi(v)\iff u >_{\psi_{>\hskip-4pt>}}v $). This is an admissible order on $\Jord_\rho(\psi)$ and  $\varphi$ preserves the order (by definition of $>_{\psi}$).
In this way $\Jord(\psi_{>\hskip-4pt>})$ dominates $\Jord(\psi)$ and by   \cite[3.1.2]{MR2767522} or   \cite[section 8]{MR3679701}, we can get all the elements of $\Pi_{\psi}$ from the elements of $\Pi_{\psi_{>\hskip-4pt>}}$ applying 
$
\Jac_{\alpha+m+1}
$
(each application of the operator $\Jac_{\alpha+m+1}$ will result in either 
 an irreducible representation or $0$). 
Observe that
\begin{multline}
L([\alpha-1,\alpha+m]\tr;\delta([\alpha,\alpha+m+1];\sigma))
\h
\\
L([-(\alpha+m),-(\alpha-1)]\tr)\rtimes\delta([\alpha,\alpha+m+1];\sigma)
\\
\h L([-(\alpha+m),-(\alpha-1)]\tr)\times [\alpha+m+1]\rtimes\delta([\alpha,\alpha+m];\sigma)
\\
\cong
 [\alpha+m+1] \times
L([-(\alpha+m),-(\alpha-1)]\tr)\rtimes\delta([\alpha,\alpha+m];\sigma).
\end{multline}
Obviously
\begin{multline}
[\alpha+m+1] \rtimes
L([\alpha-1,\alpha+m]\tr;\delta([\alpha,\alpha+m];\sigma))
\\
\h
[\alpha+m+1] \times
L([-(\alpha+m),-(\alpha-1)]\tr)\rtimes\delta([\alpha,\alpha+m];\sigma).
\end{multline}
One  sees directly that the last representation has a unique irreducible subrepresentation (showing that the multiplicity of $[\alpha+m+1] \otimes
L([-(\alpha+m),-(\alpha-1)]\tr)\otimes\delta([\alpha,\alpha+m];\sigma)$ in the Jacquet module is one). This implies 
\begin{multline}
L([\alpha-1,\alpha+m]\tr;\delta([\alpha,\alpha+m+1];\sigma))
\\
\h
[\alpha+m+1] \rtimes
L([\alpha-1,\alpha+m]\tr;\delta([\alpha,\alpha+m];\sigma)).
\end{multline}
Now Frobenius reciprocity implies that 
$$
\Jac_{\alpha+m+1}(L([\alpha-1,\alpha+m]\tr;\delta([\alpha,\alpha+m+1];\sigma)))=
L([\alpha-1,\alpha+m]\tr;\delta([\alpha,\alpha+m];\sigma)),
$$
and therefore, $L([\alpha-1,\alpha+m]\tr;\delta([\alpha,\alpha+m];\sigma))$ is in 
the A-packet of $\psi_{2\alpha+1+2m,2\alpha+1+2m}$. 

\subsubsection{Case $0\leq m<n$}  Observe that if we take any $\psi'\in\Pe\cap \Pddr$ such that $(\psi')_d=(\psi_-)_d$ instead of $\psi_-$ in \ref{theorem >1}  
  and use $\psi'$ (instead of $\psi_-)$ to define $\psi_{k,l}$, $\e_{k,l}$ and $\psi_{k,l}'$, we get exactly the same results as we have obtained in the proof so far.

Assume below $0\leq m<n$. Now in \ref{-2 leq m<n} put $\psi'=(\psi_-)\tr$ and denote the objects that correspond to $\psi_{k,l}$ and  $\e_{k,l}$ for this $\psi'$ 
 by $\psi_{k,l}''$ and $\e_{k,l}''$ (recall $\psi_{2\alpha+1+2n,2\alpha+1+2m}''=L([\alpha-1,\alpha+m]\tr;\delta([\alpha,\alpha+n];\sigma))$).
 Then 
$$
(\psi_{2\alpha+1+2n,2\alpha+1+2m}'')\tr= \psi_{2\alpha+1+2m,2\alpha+1+2n},
$$
and $\e_{2\alpha+1+2n,2\alpha+1+2m}''$ give the same diagonal restriction as $\psi_{2\alpha+1+2m,2\alpha+1+2n}'$
(defined in \ref{-1 leq n<m}). Now \ref{inv-formula-el} and Proposition \ref{tr >1 inv} imply
\begin{multline}
\pi((\psi,\e)_{2\alpha+1+2m,2\alpha+1+2n})
=
\pi((\psi_{2\alpha+1+2n,2\alpha+1+2m}'')\tr, \e_{2\alpha+1+2n,2\alpha+1+2m}'')
\\
=
\pi((\psi,\e)_{2\alpha+1+2n,2\alpha+1+2m})\tr=L([\alpha-1,\alpha+m]\tr;\delta([\alpha,\alpha+n];\sigma))\tr 
\\
=
L([\alpha-1,\alpha+n]\tr;\delta([\alpha,\alpha+m];\sigma)).
\end{multline}
\end{proof}

Note that in the proof of the above theorem we have also proved what happens with  the few additional cases where $m\geq -2$ and $n\geq -1$. We comment these mostly well known  cases briefly  in  the following

\begin{corollary} 
\label{cor-ends}
\label{}
Let $m\geq -2$ and $n\geq -1$.
\begin{enumerate}

\item
\label{cor-str-st-tr}
 The representation
$\delta([\alpha,\alpha+n];\sigma)$
(resp. $L([\alpha,\alpha+n]\tr;\sigma)$)
is  in $\Pi_\psi$ for $\psi=\psi_-\oplus E^\rho_{2\alpha-3,1}\oplus E^\rho_{2\alpha+1+2n,1}$ (resp. $\psi=\psi_-\oplus E^\rho_{1,2\alpha-3}\oplus E^\rho_{1,2\alpha+1+2n}$).

\item
\label{strange-c-s} 
For   $m,n\geq -1$,
the representations
$L([\alpha-1];\delta([\alpha,\alpha+n];\sigma))$ and
 $L([\alpha-1,\alpha+m]\tr;\sigma)$ are in A-packets.
These representations are at the end of  complementary series if $n\geq0$ (resp $m\geq0$).

\item
\label{cor-st-tr}
$\delta([\alpha,\alpha+n];\sigma)\tr=L([\alpha,\alpha+n]\tr;\sigma)$.

\item
\label{cor-st-tr+}
$L([\alpha-1];\delta([\alpha,\alpha+n];\sigma))\tr
=
L([\alpha-1,\alpha+n]\tr;\sigma)$.
\end{enumerate}
\end{corollary}

\begin{proof} The first three claims are proved in the previous theorem. It remains to consider  only (4). This is very simple to prove  by similar methods as these used in the proof of Proposition \ref{tr >1 inv}, and therefore we omit them.
\end{proof}

\section{Case of reducibility \texorpdfstring{$0$}{}}
\label{sec: 0}

In this and the following two sections we will handle the remaining reducibilities, i.e. $\alpha=0,\frac12$ and $1$, and write down A-packets and representations in them which can be considered as analogous cases for these reducibilities. It is very easy to get that they are in A-packets. We will also give  some additional information about them (formulas for the Aubert involution, and to which characters of the component groups they correspond in the case of discrete parameter).

In this section  $\rho$, $\sigma $ and $\alpha$ are as in \ref{rho and sigma}, and we assume   that  
$
\alpha=0.
$
We  fix a decomposition of $\rho\rtimes\sigma$ in \eqref{pm}. We denote by $\psi_\sigma$ the tempered elementary discrete parameter such that $\sigma\in \Pi_{\psi_\sigma}$. Applying \cite[Proposition 6.0.3]{MR2767522} to the parameter $\psi_\sigma\oplus E^\rho_{2n+1,1}\oplus E^\rho_{1,2m+1}$ (for which we know by  \ref{modifying} that it is again an A-parameter) we get directly that if
$m,n\geq 0$, then
$$
L([1,m]\tr;\delta([0,n]_\pm;\sigma))\in \Pi_{\psi_\sigma\oplus E_{2n+1,1}^\rho
\oplus E_{1,2m+1}^\rho}.
$$
In the following theorem, we give  additional information about elements of these packets in the case $m\ne n$.

\subsection{Definition of $(\psi,\e^\pm)_{k,l}$ in the case $\alpha=0$} 
\label{not 0}
Denote in this section 
$$
\psi:=\psi_\sigma,
\qquad
\psi_{k,l}:= 
 E_{ k,1}^\rho\oplus E_{1, l}^\rho,
$$
where $k,l\geq 0$ will always be  chosen to be of odd parity,
and   denote by $\e_{k,l}^\pm$ the character of the component group of $\psi_{k,l}$ which extends  $\e_\sigma$ on  $\psi$, and which is equal to $\pm1$ on remaining two elements.
Similarly as before,  we denote a pair $(\psi_{k,l},\e_{k,l}^\pm)$ by 
$
(\psi,\e^\pm)_{k,l}.
$

\subsection{On corresponding A-packets and involution}

\begin{theorem}
\label{thm-0}
Let $n,m\geq0$. Then
\begin{enumerate} 
\item
\label{packet a=0}
 $L([1,m]\tr;\delta([0,n]_{\pm};\sigma))\in \Pi_{\psi_{2n+1,2m+1}}$.
In particular, $L([1,m]\tr;\delta([0,n]_{\pm};\sigma))$ are unitarizable.
\item
\label{pm formula red 0} For $m\ne n$ we have
$
\pi((\psi,\e^\xi)_{2n+1,2m+1})=
L([1,m]\tr;\delta([0,n]_{\text{sign}(n-m)\xi};\sigma)), \ \xi\in\{\pm\}
$.
\end{enumerate}

\end{theorem}

\begin{proposition}
\label{tr 0 inv}
 Let  $\alpha=0$ and $m,n\in \Z_{\geq0}$. Denote
$$
\pi_{m,n}^\pm:=L([1,m]\tr;\delta([0,n]_\pm;\sigma)).
$$ 
Then
 $$
(\pi_{n,m}^\pm)\tr
=
\pi_{m,n}^\mp.
$$
\end{proposition}

\begin{remark} Elements of an A-packet $\psi$ of good parity which are not discrete are obtained from some suitable discrete A-packet $\psi_{>\hskip-4pt>}$ of a bigger group by a procedure described in \cite[3.1.2]{MR2767522} (see also \cite[section 8]{MR3679701}).
Here one applies Jacquet module operators to representations $\pi(\psi,\boldsymbol t, \boldsymbol \eta)$ to get elements of $\Pi_\psi$ (the result can also be  $0$). The result can depend on an admissible order that one fixes on $\jrp$. We comment below an example where one gets different results for A-packets corresponding to $m=n>1$ for an admissible order which satisfies $(\rho,m,1)>_\psi(\rho,1,m)$ and admissible order which satisfies $(\rho,1,m)>_\psi(\rho,m,1)$ (in our case $\boldsymbol t$ is the zero function, and $\boldsymbol \eta=\e^\pm$). The reason is that on $\Jord (\psi_{>\hskip-4pt>})$ we need to take a natural order. Therefore, in the case of the first admissible order, one gets 
$$
\Jac_{m+1}( L([1,m]\tr;\delta([0,m+1]_\pm;\sigma)))= L([1,m]\tr;\delta([0,m]_\pm;\sigma)),
$$
 while in the case of the second admissible order, one gets 
 $$
\Jac_{-(m+1)}( L([1,m+1]\tr;\delta([0,m]_\mp;\sigma)))= L([1,m]\tr;\delta([0,m]_\mp;\sigma)).
$$
  Note that if we  change admissible orders as above in the setting of Theorem \ref{theorem >1} for the case $m=n\geq 0$, the results there remain unchanged. The same holds for settings of Theorem  \ref{thm-1/2}.
\end{remark}

\section{Case of reducibility \texorpdfstring{$\frac12$}{}}
\label{sec: 1/2}

As in the previous   sections,  $\rho$, $\sigma $ and $\alpha$ are as in \ref{rho and sigma}, and we assume in this section  that  
$
\alpha=\tfrac12.
$
By $\psi_\sigma$ we denoted the tempered elementary discrete A-parameter such that $\sigma\in \Pi_{\psi_\sigma}$. Applying   \cite[Proposition 6.0.3]{MR2767522} to $\psi_\sigma\oplus E_{2n,1}^\rho
\oplus E_{1,2m}^\rho$ (which is also an A-parameter by \ref{modifying}), we get immediately that for
$m,n\geq 0$,
$$
L([\tfrac12, \tfrac{2m-1}2]\tr;\delta([\tfrac12,\tfrac{2n-1}2];\sigma))\in \Pi_{\psi_\sigma\oplus E_{2n,1}^\rho
\oplus E_{1,2m}^\rho}.
$$
Later  
we will  give additional information regarding these packets.

\subsection{Involution}
First we will see how these representations transform under the Aubert involution.

\begin{proposition}
\label{prop inv 1/2}
 For $m,n\geq 1$ 
 denote
\begin{equation}
\label{not pi for 1/2}
\pi_{m,n}^+:=L([\tfrac12, \tfrac{2m-1}2]\tr;\delta([\tfrac12,\tfrac{2n-1}2];\sigma)),
\ \
\pi_{m,n}^-:= L([\tfrac32, \tfrac{2m-1}2]\tr;\delta([-\tfrac12,\tfrac{2n-1}2]_-;\sigma)).
\end{equation} 
Then
\begin{equation}
\label{inv 1/2}
(\pi_{m,n}^+)\tr=\pi_{n,m}^-.
\end{equation}
\end{proposition}

\subsection{Definition of $(\psi,\e^\pm)_{k,l}$ in the case $\alpha=\frac12$}
Denote by
$$
\psi:=\psi_\sigma,
\qquad
\psi_{k,l}:= 
 E_{ k,1}^\rho\oplus E_{1, l}^\rho,
$$
where $k,l\geq 0$ will be always chosen to be of even parity,
and  denote by $\e_{k,l}^\pm$ the character of the component group of $\psi_{k,l}$ which extends  $\e_\sigma$ on  $\psi$, and which is equal to $\pm1$ on the remaining two elements.
As before,  we denote a pair $(\psi_{k,l},\e_{k,l}^\pm)$ by 
$
(\psi,\e^\pm)_{k,l}.
$

\subsection{On corresponding A-packets}
With the above notation (and  $\pi_{m,n}^\pm$ introduced in Proposition \ref{prop inv 1/2}), we have the following

\begin{theorem}
\label{thm-1/2} Let  $m,n\ge1$. Then the following holds:

\begin{enumerate}

\item
$$
L([\tfrac12, \tfrac{2m-1}2]\tr;\delta([\tfrac12,\tfrac{2n-1}2];\sigma)), \ 
L([\tfrac32, \tfrac{2m-1}2]\tr;\delta([-\tfrac12,\tfrac{2n-1}2]_-;\sigma)) 
\in
\Pi_{\psi_{2n,2m}}.
$$
In other words, $\pi_{m,n}^\pm \in \Pi_{\psi_{2n,2m}}.$ In particular, representations $\pi_{m,n}^\pm$ are unitarizable.

\item 
\begin{equation}
\label{formula 1/2 +}
\pi((\psi,\e^+)_{2n,2m})=
\begin{cases}
L([\tfrac12, \tfrac{2m-1}2]\tr;\delta([\tfrac12,\tfrac{2n-1}2];\sigma)) =\pi_{m,n}^+,&  m< n,
\\
L([\tfrac32, \tfrac{2m-1}2]\tr;\delta([-\tfrac12,\tfrac{2n-1}2]_-;\sigma)) =\pi_{m,n}^-, &  n<m.
\end{cases}
\end{equation}

\item 
\begin{equation}
\label{formula 1/2 -}
\pi((\psi,\e^-)_{2n,2m})=
\begin{cases}
L([\tfrac32, \tfrac{2m-1}2]\tr;\delta([-\tfrac12,\tfrac{2n-1}2]_-;\sigma))=\pi_{m,n}^- , &  m<n.
\\
L([\tfrac12, \tfrac{2m-1}2]\tr,\delta([\tfrac12,\tfrac{2n-1}2];\sigma))=\pi_{m,n}^+ ,&  n< m.
\end{cases}
\end{equation}

\end{enumerate}
\end{theorem}

\section{Case of reducibility at \texorpdfstring{$1$}{}}
\label{sec: 1}

Again in this section $\rho$, $\sigma $ and $\alpha$ are as in \ref{rho and sigma}, and we assume   that  
$
\alpha=1.
$
Denote by $\psi_\sigma$ the tempered elementary discrete parameter such that $\sigma\in \Pi_{\psi_\sigma}$. 
Applying
Proposition 6.0.3 of \cite{MR2767522} 
to $\psi_\sigma\oplus E_{2n+1,1}^\rho\oplus E_{1,2m+1}^\rho$ we get that  
for $m,n\geq 1$,
$$
L([1,m]\tr;\tau([0]_\pm;\delta([1,n];\sigma)))
\in \Pi_{\psi_\sigma\oplus E_{2n+1,1}^\rho
\oplus E_{1,2m+1}^\rho}.
$$
Before we give more information about these packets, we calculate the Aubert involutions of the above representations.

\subsection{Involution} We start with the following

\begin{lemma} For $n\geq 1$ we have
\begin{equation}
\label{lemma =1}
\tau([0]_{x};\delta([1,n];\sigma))\tr=
\begin{cases}
L([1,n]\tr;[0]\rtimes\sigma)
, & x=+,
\\
L([0,1],[2,n]\tr;\sigma)
, & x=-.
\end{cases}
\end{equation}
The above representations are unitarizable.
\end{lemma}

Now we have the following 

\begin{proposition}
\label{prop inv 1}
 Let $m,n\geq 1$. Denote
$$
\pi_{m,n}^\pm:=L([1,m]\tr;\tau([0]_\pm;\delta([1,n];\sigma))
,
\quad
\tau_{m,n}^-:= L([2,m]\tr;\delta([-1,n]_-;\sigma)).
$$
Then
\begin{equation}
\label{inv 1}
(\pi_{m,n}^+)\tr=\pi_{n,m}^+,  \qquad (\pi_{m,n}^-)\tr=\tau_{n,m}^- .
\end{equation}
\end{proposition}

\subsection{Definition of $(\psi,\e^\pm)_{k,l}$ and $\e^{+,-,-}_{k,l}$ in the case $\alpha=1$}
Denote in the rest of this section 
$$
\psi_{k,l}:= 
\psi
\oplus
 E_{ k,1}^\rho\oplus E_{1, l}^\rho,
$$
where $k,l\geq 0$ will be always chosen to be of odd parity.
Set
$$
\xi=\e_\sigma(\rho,1,1).
$$
Next we define  characters $\e_{k,l}^\pm$  of the component group of $\psi_{k,l}^\pm$ when $k$ and $l$ are different odd integers $>1$. They coincide with $\psi_\sigma$ on 
$\Jord(\psi_\sigma)- ((\rho,1,1))$ and satisfy
\begin{gather*}
\e_{k,l}^\pm(\rho,1,1)=\e_{k,l}^\pm(\rho,\min(k,l),\delta_{\min(k,l)})=\pm\xi, 
\\
\e_{k,l}^\pm(\rho,\max(k,l),\delta_{\min(k,l)})=\xi
\end{gather*}
(we need to assume that $\e^\pm_{k,l}$ is equal on the pair of blocks for which $E_{k',l'}^\rho=E_{k'',l''}^\rho$).
As before,  we denote such a pair $(\psi_{k,l},\e_{k,l}^\pm)$ by 
$
(\psi,\e^\pm)_{k,l}.
$
Denote
$$
\e^{+,-,-}_{k,l}
$$
\begin{gather*}
\e^{+,-,-}_{k,l}(\rho,1,1)=\xi,
\\
\e^{+,-,-}_{k,l}(\rho,\min(k,l),\delta_{\min(k,l)})=
\e^{+,-,-}_{k,l}(\rho,\max(k,l),\delta_{\min(k,l)})=-\xi.
\end{gather*}

\subsection{On corresponding A-packets}
With the above notation we have the following

\begin{theorem}
\label{thm-1} Let  $m,n\ge1$. Then the following holds:

\begin{enumerate}

\item
\label{item packet 1}
$$
L([1,m]\tr;\tau([0]_\pm;\delta([1,n];\sigma))), 
L([2,m]\tr;\delta([-1,n]_-;\sigma))
\in 
\Pi_{\psi_\sigma\oplus E_{2n+1,1}^\rho
\oplus E_{1,2m+1}^\rho}.
$$
 In other words, $\pi_{m,n}^\pm,\tau_{m,n}^- \in \Pi_{\psi_{2n+1,2m+1}}.$ In particular, representations $\pi_{m,n}^\pm$ and $\tau_{m,n}^-$ are unitarizable.

\item 
\begin{equation}
\label{formula 1 +}
\pi((\psi,\e^+)_{2n+1,2m+1})=
\begin{cases}
L([1,m]\tr;\tau([0]_{-};\delta([1,n];\sigma))) =\pi_{m,n}^-,&  m< n,
\\
L([2,m]\tr;\delta([-1,n]_-;\sigma)) =\tau_{m,n}^-, &  n<m.
\end{cases}
\end{equation}

\item 
\begin{equation}
\label{formula 1 -}
\pi((\psi,\e^-)_{2n+1,2m+1})=
L([1,m]\tr;\tau([0]_{+};\delta([1,n];\sigma))) =\pi_{m,n}^+ , \ \ \ 
  m\ne n.
\end{equation}

\item 
\begin{equation}
\label{formula 1 +--}
\pi((\psi,\e^{+--})_{2n+1,2m+1})=
\begin{cases}
L([2,m]\tr;\delta([-1,n]_-;\sigma)) =\tau_{m,n}^-,
&  m< n,
\\
L([1,m]\tr;\tau([0]_{-};\delta([1,n];\sigma))) =\pi_{m,n}^-, &  n<m.
\end{cases}
\end{equation}

\end{enumerate}
\end{theorem}

\section{On irreducible unitarizable subquotients at critical points}
\label{sec: CR<4}

\begin{definition}
\label{def-critical}
Let $\rho_1,\dots,\rho_k\in\Cusp$ and let  $\sigma$ be an irreducible cuspidal representation of a classical group.
Assume that for any  $i$ we have 
\begin{enumerate}

\item
$\rho_i^u\cong (\rho_i^u)\check{\ }$; 

\item
\label{Z-segment}
 the set
$
\{e(\rho_j):\rho_j^u\cong \rho_i^u\}
$
is a $\Z$-segment in $\frac12\Z$ (possibly with multiplicities); 

\item
 the $\Z$-segment in (\ref{Z-segment}) contains the reducibility exponent $\alpha_{ \rho_i^u,\sigma}$.
 
 \end{enumerate}
Then, we say that the representation   $\rho_1\times\dots\times \rho_k\rtimes\sigma$ is of critical type. If additionally   $\pi$ is an irreducible subquotient of $\rho_1\times\dots\times \rho_k\rtimes\sigma$, then we also say  that $\pi$ is of critical type.
\end{definition}

The aim of this section is to prove the following

\begin{theorem} 
\label{cr leq 3}
Let $\pi$ be an irreducible unitarizable subquotient of a representation
$$
\rho_1\times\dots\times \rho_k\rtimes\sigma,\qquad k\leq3
$$
of critical type. Then $\pi$ is contained in an A-packet.
\end{theorem}

\begin{proof} First we  recall   some simple general facts which will  considerably shorten the proof of the theorem.\subsubsection{Some simple remarks about A-packets}
\label{ECS}
\begin{enumerate}
\item Each irreducible tempered representation is an element of some A-packet (with tempered A-parameter).

\item If $\pi$ is an element of an elementary discrete A-packet, then $\pi\tr$ is also  an element of an elementary discrete A-packet.

\item Each irreducible cotempered representation is contained in an A-packet (with co\-tempered A-parameter).

\end{enumerate}
For coranks 0 and 1 the  theorem  follows  directly from remarks in \ref{ECS}. It remains to consider   coranks 2 and 3. 
We will consider below only the cases which are not covered by remark \ref{ECS}.   We will also prove the theorem in the case when all $\rho_i^u$ are the same, denoted by $\rho$ (the proof in the other case is very simple, and we omit it here). We fix an irreducible cuspidal representation $\sigma$ of a classical group. We assume that $\sigma=\pi(\psi_\sigma,\e_\sigma)$ for some $\psi_\sigma\in\Pe\cap \Pddr$. Denote $\alpha=\alpha_{\rho,\sigma}$ (as usual). If we have some $\psi\in\Pe$, and write $\jrp=((a_1,b_1),\dots,(a_k,b_k)) $, then we will always assume that the enumeration satisfies $\max(a_1,b_1)\leq \dots \leq \max(a_k,b_k)$.

Below we will consider  exponents $(x_1,\dots,x_k)$, $k=2$ or $3$,  the representation $\nu^{x_k}\rho\times\dots\times \nu^{x_1}\rho\rtimes\sigma$ of critical type, and irreducible unitarizable subquotients of it.
We will  give precise references about where these representations were considered in \cite{MR-T-CR3}, and denote their irreducible unitarizable subquotients in the same way as in  \cite{MR-T-CR3} (therefore, we will not recall here  this notation).

The arguments below are usually  simple (and we have already used them in the previous part of the paper). Therefore, we will  only sketch them very  briefly below.

When we  have a parameter $(\psi',\e')$ as below, and when we  get a new parameter $(\psi'',\e'')$ by replacing $\rabp\in \psi'$ by $\rabpp$, then we will always assume that $\e'\rab'=\e''\rabpp$ and that $\e'$ and $\e''$ coincide on remaining blocks. 

Also if we  get $(\psi'',\e'')$ from $(\psi',\e')$ by replacing some elements $\rab$ with $(\rho,b,a)$, then we will assume $\e''\rab=\e'(\rho,c,d)$ if $\max(a,b)=\max(c,d)$.

\subsection{Corank 2}

\subsubsection{Case $(\alpha-1,\alpha), \alpha>1$ {\rm (3.4.3 of \cite{MR-T-CR3})}} 
\label{a-1,a}
 Here all 4 irreducible subquotients are unitarizable. One is square integrable, and another is its Aubert involution. Therefore, we need to consider only representations 
$$
\pi_2:=L([\alpha-1];\delta([\alpha];\sigma)), \quad \pi_3:=L([\alpha-1],[\alpha];\sigma),
$$
where $\pi_2\tr=\pi_3$. Both above representations are contained in A-packets by \eqref{strange-c-s}  of Corollary \ref{cor-ends} (in the corollary, consider the case of  $n=0$ and $m=0$, respectively).

\subsubsection{Case $(0.1), \alpha=0$ {\rm (3.4.6 of \cite{MR-T-CR3})}} Here all 5 irreducible subquotients are unitarizable. Two of them are square integrable, and  another two are their Aubert duals.  Therefore we need to consider  only
$$
\pi_2:=L([0,1];\sigma).
$$
Let $\psi:=\psi_\sigma\oplus E_{2,2}^\rho$. Then $\pi_2\in \Pi_\psi$ by Proposition 6.0.3 of \cite{MR2767522} (construction "$L$-packet  inside A-packet").

\subsection{Corank 3}

\subsubsection{Case $(\alpha-1,\alpha,\alpha+1), \alpha>1$ {\rm (4.5 of \cite{MR-T-CR3})}}  Here we have 4 irreducible unitarizable  subquotients. One of them  is square integrable, and another is its Aubert involution. Therefore,  we need to consider  the following 
representations 
\begin{gather*}
\pi_3:=L([\alpha-1];\delta([\alpha,\alpha+1];\sigma)),
\quad
\pi_4:=L([\alpha+1],[\alpha],[\alpha-1];\sigma).
\end{gather*}
Both above representations are contained in A-packets by \eqref{strange-c-s}  of Corollary \ref{cor-ends} (in the corollary, consider the case of  $n=1$ and $m=1$, respectively).

\subsubsection{$(\alpha-1,\alpha,\alpha), \alpha>1$ {\rm (4.6 of \cite{MR-T-CR3})}}   Here only one irreducible subquotient is unitarizable:  
$$
\pi_0:=L([\alpha-1], [\alpha];\delta([\alpha];\sigma)).
$$
The above representation is contained in an A-packet by \eqref{packet a > 3/2}  of Theorem \ref{theorem >1} (in the theorem, consider the case of  $n=0$ and $m=0$; recall that C. M\oe glin has shown that this representation is in an A-packet in Appendix A of \cite{MR-T-CR3}).

\subsubsection{Case $(\tfrac12,\tfrac12,\tfrac32), \alpha=\tfrac32$ {\rm (4.7.2 of \cite{MR-T-CR3})}}    Here all 8 irreducible subquotients are unitarizable. Two of them are tempered, while another two are cotempered. Therefore it remains  to consider   representations
\begin{gather*}
\pi_5:=L([\tfrac32];\delta([-\tfrac12,\tfrac12])\rtimes\s), \ \
\pi_6:=L([\tfrac12],[\tfrac12];\delta([\tfrac32];\s)),\\
\pi_7:=L([-\tfrac12,\tfrac32];\s), \ \
\pi_8:=L([\tfrac12];\delta_{\spsi}([\tfrac12],[\tfrac32];\s)).
\end{gather*}
For $\pi_5$, consider 
$\psi'_\sigma$, 
which we get 
from $\psi_\sigma$  by replacing $(2,1) $ with $(1,2)$ in $\jrp$. Now increase $(1,2)$ to $(1,4)$ and denote new parameter by $\psi'$. We get $L([\tfrac32];\sigma)$ in the packet of $\psi'$. Now add $(2,1),(2,1)$ to $\Jord_\rho(\psi')$. Now applying   \cite[Proposition 5.3]{MR2504024} we get  that $\pi_5$ is in this new packet.

For $\pi_6$, increase $(2,1)$ to $(4,1)$ in $\Jord_\rho(\sigma)$ and denote this packet b $\psi'$. Then $\delta([\tfrac32];\sigma)$ is in the new packet. Now add $(1,2),(1,2)$ to $\Jord_\rho(\psi')$, and we get $\pi_8$ in the packet of this new parameter.

Observe that $\pi_7\in\Pi_{\psi}$, where $\psi:=\psi_\sigma\oplus E_{3,2}^\rho$.

For $\pi_8$, recall that  $(2,1)\in \Jord_\rho(\psi_\sigma)$. Then increasing $(2,1)$ to $(6,1)$ we  get $\delta([\tfrac32,\tfrac52];\s)$ in the packet. Adding $(2,1)$ and then replacing it with $(4,1)$, we  get (in two steps) that  $\delta_{\spsi}([\tfrac12,\tfrac32],[\tfrac32,\tfrac52];\s)$. Adding $(1,2)$ to the previous packet, we  get $L([\frac12];\delta_{\spsi}([\tfrac12,\tfrac32],[\tfrac32,\tfrac52];\s))$ in the packet. This representation is (by our construction) in $\Pi_\psi$, where $\jrp=((6,1),(4,1),(1,2))$. Put a standard order on $\jrp $. Denote by $\psi'$ the A-parameter obtained from $\psi$ by changing $\jrp$ to $\jr(\psi')=(4,1),(2,1),(1,2)$. 
 Consider a  standard order on $\jr(\psi')$ satisfying $(2,1)>_{\psi'}(1,2)$, and let $\varphi:\jrp\rightarrow \jr(\psi')$ be a standard bijection which preserves order. Then $\jrp$ dominates $\jr(\psi')$ with respect to $>_{\psi'}$.
By   \cite[3.1.2]{MR2767522} or   \cite[section 8]{MR3679701},  $\Jac_{[\frac52]}\circ\Jac_{[\frac32]}(L([\frac12];\delta_{\spsi}([\tfrac12,\tfrac32],[\tfrac32,\tfrac52];\s)))$, we  get an element of the packet of $\psi'$ or $0$.
To compute the last representation, observe that 
$$
L([\tfrac12];\delta_{\spsi}([\tfrac12,\tfrac32],[\tfrac32,\tfrac52];\s))
\h
[\tfrac32]\times [-\tfrac12]\rtimes\delta_{\spsi}([\tfrac12],[\tfrac32,\tfrac52];\s),
$$
and that the last representation has a unique irreducible subrepresentation. This implies
$$
L([\tfrac12];\delta_{\spsi}([\tfrac12,\tfrac32],[\tfrac32,\tfrac52];\s))
\h
[\tfrac32]\rtimes L( [\tfrac12];\delta_{\spsi}([\tfrac12],[\tfrac32,\tfrac52];\s)),
$$
which easily implies
$$
\Jac_{[\frac32]}(L([\tfrac12];\delta_{\spsi}([\tfrac12,\tfrac32],[\tfrac32,\tfrac52];\s)))
=
 L( [\tfrac12];\delta_{\spsi}([\tfrac12],[\tfrac32,\tfrac52];\s)).
$$
Observe that
$$
\delta_{\spsi}([\tfrac12],[\tfrac32,\tfrac52];\s))
\h
[\tfrac12]\times [\tfrac52]\rtimes\delta([\tfrac32];\s)
\cong
 [\tfrac52]\times[\tfrac12]\rtimes\delta([\tfrac32];\s).
$$
Since the last representation has a unique irreducible subrepresentation, we get 
$$
\delta_{\spsi}([\tfrac12],[\tfrac32,\tfrac52];\s))
\h
 [\tfrac52]\rtimes\delta_{\spsi}([\tfrac12],[\tfrac32];\s).
$$
Now
$$
 L( [\tfrac12];\delta_{\spsi}([\tfrac12],[\tfrac32,\tfrac52];\s))
 \h
  [-\tfrac12]\times [\tfrac52]\rtimes\delta_{\spsi}([\tfrac12],[\tfrac32];\s)
  \cong
  [\tfrac52]\times [-\tfrac12]\rtimes\delta_{\spsi}([\tfrac12],[\tfrac32];\s).
$$
Since the last representation has a unique irreducible subrepresentation, we get
$$
 L( [\tfrac12];\delta_{\spsi}([\tfrac12],[\tfrac32,\tfrac52];\s))
 \h
   [\tfrac52]\rtimes L([\tfrac12];\delta_{\spsi}([\tfrac12],[\tfrac32];\s)).
$$
This implies $\Jac_{[\frac52]}( L( [\tfrac12];\delta_{\spsi}([\tfrac12],[\tfrac32,\tfrac52];\s)))=
L([\tfrac12];\delta_{\spsi}([\tfrac12],[\tfrac32];\s))$, and completes the proof that $\pi_8$ is in an A-packet.

\subsubsection{$(\alpha-2,\alpha-1,\alpha), \alpha>2$ {\rm (4.8.1 of \cite{MR-T-CR3}) }} 
\label{a-2,a-1,a}
Here all 8 irreducible subquotients are unitarizable. They are
\begin{gather*}
\pi_1=\delta_{\spsi}([\alpha-2],[\alpha-1],[\alpha];\s),\ \
\pi_2=L([\alpha-2];\delta_{\spsi}([\alpha-1],[\alpha];\s)),\\
\pi_3:=L([\alpha-1],[\alpha-2];\delta([\alpha];\s)),\ \
\pi_4:=L([\alpha-2,\alpha-1];\delta([\alpha];\s)),\\
\pi_5:=L([\alpha],[\alpha-1],[\alpha-2];\s),\ \
\pi_6:=L([\alpha],[\alpha-2,\alpha-1];\s),\\
\pi_7:=L([\alpha-1,\alpha],[\alpha-2];\s),\ \
\pi_8:=L([\alpha-2,\alpha];\s).
\end{gather*}
We have  $\pi_1^t=\pi_8, \ \pi_2^t=\pi_7, \ \pi_3^t=\pi_6, \ \pi_4^t=\pi_5$. Since $\pi_1$ is tempered, and $\pi_8$  cotempered, it remains  to consider  6 representations. 

For $\pi_7$ observe  that $\delta_{\spsi}([\alpha-1],[\alpha];\s)$ is in the A-packet corresponding to $\psi_1$, where we get
$\psi_1$ 
from 
$\psi_\sigma$  by replacing $(2\alpha-3,1),(2\alpha-1,1)$ with $(2\alpha-1,1),(2\alpha+1)$ in $\Jord_\rho(\psi_\sigma)$.
Note that $\Jord_\rho(\psi_1)$ ends with 
$(2\alpha-5,1),(2\alpha-1,1),(2\alpha+1,1)$. Now $L([\alpha-1,\alpha];\s)$ (which is the Aubert dual of previous discrete series by (3) of Proposition 3.7 in \cite{MR-T-CR3}) is in the A-packet of $\psi_1\tr$ and $\Jord_\rho(\psi_1\tr)$ ends with $(1,2\alpha-5),(1,2\alpha-1),(1,2\alpha+1)$. Increasing $(1,2\alpha-5)$ to $(1,2\alpha-3)$, we get $\pi_7$ in the new A-packet. Since the last A-packet is discrete and elementary,  $\pi_2$ is also in an A-packet.

For $\pi_5$, first observe that    $\sigma\in\Pi_{\psi_\sigma'}$, where one gets $\psi'_\sigma$ from $\psi_\sigma$ by replacing $(2\alpha-3,1) $ with $(1,2\alpha-3)$ in $\jrp$, i.e. 
 $\jr(\psi'_\sigma)=\{\dots,(1,2\alpha-3),(2\alpha-1,1)\}.$ One defines a new A-parameter $\psi$ by increasing  the last block by 2, and then the previous block also by 2 (now $\jrp$ ends with
 $(2\alpha-5,1),(1,2\alpha-1),(2\alpha+1,1))$, and gets $L([\alpha-1];\delta([\alpha];\sigma))\in \Pi_\psi$.
This is an elementary discrete packet. Therefore  $L([\alpha-1],[\alpha];\sigma)=L([\alpha-1];\delta([\alpha];\sigma))\tr$ is in an elementary discrete packet of $\psi\tr$, and $\psi\tr$ ends with $(1,2\alpha-5),(2\alpha-1,1),(1,2\alpha+1).$ Replace $(1,2\alpha-5)$ with $(1,2\alpha-3)$ in $\psi\tr$. Then in this new packet we have the unique irreducible subrepresentation of $[-(\alpha-2)]\rtimes L([\alpha-1],[\alpha];\sigma)$. It is easy to show that this unique irreducible subrepresentation is $\pi_5$. Therefore, $\pi_5$ is in an A-packet.
 Further $\pi_4$ is in an A-packet since $\pi_5$ is in an elementary discrete A-packet (and $\pi_4\tr=\pi_5$).

For $\pi_3$  consider 
$\psi'_\sigma$ which we get 
from 
$\psi_\sigma$  by replacing $(2\alpha-5,1),(2\alpha-3,1)$ with $(1,2\alpha-5),(1,2\alpha-3)$ in $\Jord_\rho(\psi_\sigma)$. Then $\Jord_\rho(\psi_\sigma')$ ends with $(1,2\alpha-5),(1,2\alpha-3),(2\alpha-1,1)$. Now we proceed in the usual way (increasing each of these blocks by  2), and we get $\pi_3$ in the packet. Further $\pi_6$ is in an A-packet since $\pi_3$ is in an elementary discrete A-packet (and $\pi_3\tr=\pi_6$).

\subsubsection{Case $(0,1,2), \alpha=2$ {\rm (4.8.2 of \cite{MR-T-CR3})}}  Here all 8 irreducible subquotients are unitarizable. Two of them are tempered, while another two are cotempered. Therefore it remains  to consider   representations
\begin{gather*}
\pi_5= L([1];[0]\rtimes\delta([2];\sigma)), \quad \pi_6=L([2],[0,1];\sigma),\\
\pi_7= L([0,1];\delta([2];\sigma)),\quad \pi_8= L([2],[1];[0]\rtimes\sigma),
\end{gather*}
where $\pi_5\tr=\pi_6$ and $\pi_7\tr=\pi_8$.

For $\pi_5$ and $\pi_7$, recall that by \ref{a-1,a}, $L([1];\delta([2];\sigma))$ is in $\Pi_{\psi}$ for some A-parameter $\psi$. Now each irreducible subquotient of $[0]\rtimes L([1];\delta([2];\sigma))$ (it is also a subrepresentation) is in the packet of $\psi\oplus E_{1,1}^\rho\oplus E_{1,1}^\rho$. One of them is $\pi_5$ (apply \cite[Proposition 5.3]{MR2504024}). For another one, observe that $([0]\rtimes L([1];\delta([2];\sigma)))\tr=[0]\rtimes L([1];\delta([2];\sigma))\tr=[0]\rtimes L([2],[1];\sigma)$, and that here  $\pi_8$
 is a subquotient (again apply \cite[Proposition 5.3]{MR2504024}). Then $\pi_7=\pi_8\tr$  is a subquotient of $[0]\rtimes L([1];\delta([2];\sigma))$. Therefore, $\pi_7$ is also in an A-packet, as well as $\pi_5$.

For $\pi_6$ and $\pi_8$, recall that by \ref{a-1,a}, $L([2],[1];\sigma)$ is in $\Pi_{\psi}$ for some A-parameter $\psi$. Now each irreducible subquotient of $[0]\rtimes L([2],[1];\sigma)$  is in the packet of $\psi\oplus E_{1,1}^\rho\oplus E_{1,1}^\rho$. One of them is $\pi_8$ (by \cite[Proposition 5.3]{MR2504024}). For other one, observe that $([0]\rtimes L([2],[1];\sigma))\tr=[0]\rtimes L([2],[1];\sigma)\tr=[0]\rtimes L([1];\delta([2];\sigma))$, and that here  $\pi_5$
 is subquotient (by \cite[Proposition 5.3]{MR2504024}). Then $\pi_6=\pi_5\tr$  is a subquotient of $[0]\rtimes L([2],[1];\sigma)$. Therefore, $\pi_6$ is also in an A-packet, as well as $\pi_8$.

\subsubsection{Case $(0,1,1), \alpha=1$ {\rm  (5.2 of \cite{MR-T-CR3})}}  
Here all 7 irreducible subquotients are unitarizable. Two of them are tempered, while another two are cotempered. Therefore it remains  to consider   representations 
\begin{gather*}
\pi_1=L([0,1] ,[1] ;\s ), \ \pi_3=L([0,1] ;\delta([1] ;\s )),\\
\pi_4^{+}=L([1] ;\tau([0] _+;\delta([1] ;\s))),
\end{gather*}
where $\pi_1$ and $\pi_3$ are dual. 

For $\pi_1$ (resp. $\pi_3$) consider $\psi$ obtained from $\psi_\sigma$ 
 by replacing $(1,1)$  with the pair $(1,3), (2,2)$ (resp. $(3,1), (2,2)$) in $\Jord_\rho(\psi_\sigma)$. Now $\pi_1$ (resp. $\pi_3$) is in the $L$-packet inside $\Pi_\psi$ (by Proposition 6.0.3) of \cite{MR2767522}).

For $\pi_4^+$ consider $\psi:=\psi_\sigma\oplus E_{3,1}^\rho\oplus E_{1,3}^\rho$. 
One easily shows that $\pi_4$ is in the the $L$-packet inside $\Pi_\psi$ .

\subsubsection{Case $(\tfrac12,\tfrac12,\tfrac32), \alpha=\tfrac12$ {\rm (5.4 of \cite{MR-T-CR3}) }}
Here we have 8 irreducible unitarizable  subquotients (and two non-unitarizable). Two of them are  square integrable, and another two cotempered. Therefore,  we need to consider  the following
\begin{gather*}
\pi_3=L([-\tfrac12,\tfrac32];\s),\ \
\pi_4=L([\tfrac12,\tfrac32];\delta([\tfrac12];\s)),\\
\pi_7=L([\tfrac12];\delta([\tfrac12,\tfrac32];\s)),\ \
\pi_8=L([\tfrac32];\delta([-\tfrac12,\tfrac12]_-;\s)),
\end{gather*}
where $\pi_3\tr=\pi_4$ and $\pi_7\tr=\pi_8$.

Theorem \ref{thm-1/2} implies that $\pi_7$ and $\pi_8$ are in A-packets.
For $\pi_3$ (resp. $\pi_4$) consider $\psi:=\psi_\sigma\oplus E_{3,2}^\rho$ (resp. $\psi:=\psi_\sigma\oplus E_{2,3}^\rho$). 
One directly sees  that $\pi_3$ (resp. $\pi_4$) is in  the $L$-packet inside the A-packet $\Pi_\psi$ .

\subsubsection{Case $(\tfrac12,\tfrac12,\tfrac12), \alpha=\tfrac12$ {\rm (5.5 of \cite{MR-T-CR3})}} 
Here all 5 irreducible subquotients are unitarizable. One of them is  tempered, and another one cotempered. Therefore,  we need to consider  the following representations 
\begin{gather*}
\pi_2=[\tfrac12]\rtimes\delta([-\tfrac12,\tfrac12]_-;\s),
\pi_3=[\tfrac12]\rtimes L([\tfrac12];\delta([\tfrac12];\s)),\\
\pi_5=L([\tfrac12];\delta([-\tfrac12,\tfrac12]_+;\s)),
\end{gather*}
where $\pi_2$ and $\pi_3$ are dual.

Representations $\pi_2$ and $\pi_5$ are in the $L$-packet inside the A-packet of $\psi_\sigma \oplus E_{1,2}^\rho \oplus E_{2,1}^\rho \oplus E_{2,1}^\rho$.
The representation $\pi_3$ is in the $L$-packet inside the A-packet of $\psi_\sigma \oplus E_{2,1}^\rho \oplus E_{1,2}^\rho \oplus E_{1,2}^\rho$.

\subsubsection{Case $(0,1,1), \alpha=0$  {\rm (6.2 of \cite{MR-T-CR3})}} Here 6 irreducible subquotients are unitarizable. Two of them are  tempered, and another two cotempered. Therefore,  we need to consider  the following representations 
$$
\pi_3^\pm=L([1];\delta([0,1]_\pm;\s))
$$
(here $(\pi_3^+)\tr=\pi_3^-$). The above representation is contained in an A-packet by \eqref{packet a=0}  of Theorem \ref{thm-0} (consider  the case of  $m=n=1$ in the corollary).

\subsubsection{Case $(0,0,1), \alpha=0$ {\rm (6.3 of \cite{MR-T-CR3})}}
   Here all 6 irreducible subquotients are unitarizable. Two of them are  tempered, and another two cotempered. Therefore,  we need to consider  the following representations  
$$
\pi_2^\pm= L([0,1];\delta([0]_\pm;\s)).
$$
Here $(\pi_2^+)\tr=\pi_2^-$.

Representations $\pi_2^\pm$  are in the $L$-packet inside the A-packet of 
$\psi_\sigma \oplus E_{2,2}^\rho \oplus E_{1,1}^\rho \oplus E_{1,1}^\rho$.
\end{proof}

\section*{Appendix: Some  complementary series of A-class}
\label{sec: intermediate}

Complementary series form a considerable part of unitary duals of reductive groups. Among them, the  simplest ones are one-dimensional complementary series, which in the case of classical groups are of the form
\begin{equation}
\label{cs}
\nu^x\sigma\rtimes\pi , \quad 0< x<\beta,
\end{equation}
where 
$\sigma$ and $\pi$ are irreducible unitarizable representations of a general linear and a classical group, respectively, 
such that all representations $\nu^x\sigma\rtimes\pi ,  0\leq x<\beta$, are irreducible, and such that $\nu^\beta\sigma\rtimes\pi$ is reducible.

Observe that  for parameterising the continuous family of complementary series  \eqref{cs}, it is enough to
know lower and upper bounds of the complementary series, i.e.   $\sigma$ and $\pi$ (such that $\sigma\rtimes\pi$ is irreducible), and further, the first reducibility exponent $\beta$.

C. M\oe glin mentioned to us that it is possible that some  complementary series representations can be of A-class. We present below an example of this type.
Below $\rho, \sigma$ and $\alpha$ are as in  section \ref{rho and sigma}.

\begin{lemma}
Let $\alpha\geq 1$, $x\geq 0$ and $\alpha-x\in\Z_{>0}$. Then $[x]\rtimes\sigma$ is in an A-packet (we already know that for $x=\alpha$, both irreducible subquotients are in A-packets).
\end{lemma}

\begin{proof} If $x=0$, then we are in the tempered situation, and the claim obviously holds (the A-parameter is $\psi_\sigma\oplus E_{1,1}^\rho \oplus E_{1,1}^\rho$). Therefore, we suppose $x>0$, which implies $\alpha>1$.

First we show  that $[\alpha-1]\rtimes\sigma$ is in an A-packet if $\alpha>1$. Denote by $(\psi_\sigma',\e_\sigma')$ the parameter obtained from $(\psi_\sigma,\e_\sigma)$ deforming 
$E_{2\alpha-3,1}^\rho$ to $E_{1,2\alpha-3}^\rho$ 
(then
$
\jr(\psi_\sigma')$ ends with $
(1,2\alpha-3),(2\alpha-1,1)
$). 
We have $\sigma=\pi(\psi_\sigma',\e_\sigma')$.

Denote by $(\psi_1,\e_1)$ the parameter obtained from $(\psi_\sigma',\e_\sigma')$  by deforming $E_{2\alpha-1,1}^\rho$   to $E_{2\alpha+1,1}^\rho$ (now
$
\jr(\psi_1)$ ends with $
(1,2\alpha-3),(2\alpha+1,1)
$). 
Then $\pi(\psi_1,\e_1)=\delta([\alpha];\sigma)$. 

Let $(\psi_2,\e_2)$ be obtained from $(\psi_1,\e_1)$ by  deforming $E_{1,2\alpha-3}^\rho$   to $E_{1,2\alpha-1}^\rho$ (now
$
\jr(\psi_1)$ ends with $(1,2\alpha-1),(2\alpha+1,1)
$). 
Then
$\pi(\psi_2,\e_2)=L([\alpha-1];\delta([\alpha];\sigma))$.
Denote by $>_{\psi_2}$ the standard order on $\jr(\psi_2)$.

Let  $\psi_3$ be the  A-parameter obtained from $\psi_2$ by replacing $(1,2\alpha+1)$ with $(1,2\alpha-1)$ (now
$
\jr(\psi_3)
$
ends with $(2\alpha-1,1),(1,2\alpha-1)$).
Denote by  $>_{\psi_3}$ on $\jr(\psi_3)$  standard order which satisfies
$$
(2\alpha-1,1)>_{\psi_3}(1,2\alpha-1)
$$
  ($>_{\psi_3}$ is an admissible order, but not natural; $\psi_3$ is a multiplicity one parameter, but not discrete).

We denote by  $\varphi:\Jord_\rho(\psi_2)\rightarrow \Jord_\rho(\psi_3)$ the standard bijection which  preserves order. This implies that it carries
$$
(2\alpha+1,1)\mapsto (2\alpha-1,1)
$$
(on the remaining elements it is the identity).
Now $\Jord(\psi_2)$ dominates $\Jord(\psi_3)$ with respect to $>_{\psi_3}$. Here we need to consider  the matrix $X_{(\rho,A,B,\zeta_{a,b})}^{>\hskip-4pt>}$ (defined in section 5 of \cite{MR3679701}), which is in our case a $1\times1$ matrix 
$
X_{(\rho,\alpha-1,0,1)}^{>\hskip-4pt>}=[\alpha].
$
We  get the elements of $\Pi_{\psi_3}$ from $ \Jord(\psi_2)$ applying 
$
\Jac_{\alpha}
$
to each element of $\Pi_{\psi_2}$
(the result is always   either an irreducible representation or $0$). 
Observe that
\begin{multline}
L([\alpha-1];\delta([\alpha];\sigma))
\h
[-(\alpha-1)]\rtimes\delta([\alpha];\sigma)
\\
\h [-(\alpha-1)]\times [\alpha]\rtimes\sigma
\cong
 [\alpha] \times
[-(\alpha-1)]\rtimes\sigma\cong [\alpha]\times [\alpha-1]\rtimes\sigma.
\end{multline}
Now Frobenius reciprocity implies that 
$
\Jac_{\alpha}(L([\alpha-1];\delta([\alpha];\sigma)))=
[\alpha-1]\rtimes \sigma.
$
Therefore, $[\alpha-1]\rtimes\sigma$ is in 
the A-packet of $\psi_{3}$.

\medskip

In a similar way we  show next that $[\alpha-2]\rtimes\sigma$ is in an A-packet if $\alpha>2$.
Denote now by $(\psi_\sigma',\e_\sigma')$ the parameter obtained from $(\psi_\sigma,\e_\sigma)$ by deforming $E_{2\alpha-5,1}^\rho$ to $E_{1,2\alpha-5}^\rho$
(then
$
\jr(\psi_\sigma')$ ends with $(1,2\alpha-5),(2\alpha-3,1),(2\alpha-1,1)
$). 
We have $\sigma=\pi(\psi_\sigma',\e_\sigma')$.

Denote by $(\psi_1,\e_1)$ the parameter obtained from $(\psi_\sigma',\e_\sigma')$ by  deforming $E_{2\alpha-1,1}^\rho$   to $E_{2\alpha+1,1}^\rho$ and then $E_{2\alpha-3,1}^\rho$   to $E_{2\alpha11,1}^\rho$ (now
$
\jr(\psi_1)$ ends with $
(1,2\alpha-5),(2\alpha-1,1),(2\alpha+1,1)
$). 
We get directly that $\pi(\psi_1,\e_1)=\delta_{\spsi}([\alpha-1],[\alpha];\sigma)$. 

Let $(\psi_2,\e_2)$ be obtained from $(\psi_1,\e_1)$  by deforming $E_{1,2\alpha-5}^\rho$   to $E_{1,2\alpha-3}^\rho$ (now
$
\jr(\psi_1)$ ends with $(1,2\alpha-3),(2\alpha-1,1),(2\alpha+1,1)
$). 
Then
$\pi(\psi_2,\e_2)=L([\alpha-2];\delta_{\spsi}([\alpha-1],[\alpha];\sigma))$.
Denote by $>_{\psi_2}$ the standard order on $\jr(\psi_2)$.

Denote by $\psi_3$ A-parameter obtained from $\psi_2$ by replacing $(2\alpha-1,1)$ by $(2\alpha-3,1)$ and then $(2\alpha+1,1)$ by $(2\alpha-1,1)$ (now
$
\jr(\psi_3)
$
ends with $(2\alpha-3,1),(1,2\alpha-3),(2\alpha-1,1)$).
Denote by  $>_{\psi_3}$ 
the
standard order on $\jr(\psi_3)$ which satisfies
$$
(2\alpha-3,1)>_{\psi_3}(1,2\alpha-3).
$$
Let  $\varphi:\Jord_\rho(\psi_2)\rightarrow \Jord_\rho(\psi_3)$ be the standard bijection which  preserves order. This implies that it carries
$$
(2\alpha-1,1)\mapsto (2\alpha-3,1), \qquad (2\alpha+1,1)\mapsto (2\alpha-1,1), 
$$
(on the remaining elements it is the identity).
Now $\Jord(\psi_2)$ dominates $\Jord(\psi_3)$ with respect to $>_{\psi_3}$. Here we need to consider  the matrices $X_{(\rho,A,B,\zeta_{a,b})}^{>\hskip-4pt>}$, which  in our case are $1\times1$ matrices 
$
X_{(\rho,\alpha-2,0,1)}^{>\hskip-4pt>}=[\alpha-1]$ and $ X_{(\rho,\alpha-1,0,1)}^{>\hskip-4pt>}=[\alpha]
$.
We need to apply them in descending order to $L([\alpha-2];\delta_{\spsi}([\alpha-1],[\alpha];\sigma))$, i.e. we need to apply $\Jac_\alpha\circ\Jac_{\alpha-1}$ to the last representation (and we will get either 0 or an element of $\Pi_{\psi_3}$).

Now we will compute the action of the above operator on $L([\alpha-2];\delta_{\spsi}([\alpha-1],[\alpha];\sigma)$.
Observe that
\begin{multline}
L([\alpha-2];\delta_{\spsi}([\alpha-1],[\alpha];\sigma)
\h
[-(\alpha-2)]\rtimes\delta_{\spsi}([\alpha-1],[\alpha];\sigma)
\\
\h [-(\alpha-2)]\times[\alpha-1]\rtimes\delta([\alpha];\sigma)
\cong
[\alpha-1]\times [-(\alpha-2)]\rtimes\delta([\alpha];\sigma).
\end{multline}
Note that $[-(\alpha-2)]\rtimes\delta([\alpha];\sigma)$ is irreducible. Now Frobenius reciprocity implies that 
$$
\Jac_{\alpha-1}(L([\alpha-2];\delta_{\spsi}([\alpha-1],[\alpha];\sigma)))=[-(\alpha-2)]\rtimes\delta([\alpha];\sigma).
$$
Further
$$
[-(\alpha-2)]\rtimes\delta([\alpha];\sigma)
\h
[-(\alpha-2)]\times[\alpha]\times\sigma
\cong 
[\alpha]\times [-(\alpha-2)]\times\sigma
\cong
[\alpha]\times [\alpha-2]\times\sigma,
$$
and one directly concludes that $\Jac_{\alpha}([-(\alpha-2)]\rtimes\delta([\alpha];\sigma))=[\alpha-2]\times\sigma$.
Therefore, $[\alpha-2]\rtimes\sigma$ is in 
the A-packet of $\psi_{3}$.

Continuing this procedure, we complete the proof  of the lemma.
\end{proof}

\begin{definition}
\label{def-prim} Suppose that an irreducible representation $\pi$ of a classical group $S_n$ is in an A-packet, and that there do not exist  Speh representations $\tau_1,\dots,\tau_k$ and an irreducible representation $\pi_0$ of a classical group $S_m$ with $m<n$,  contained in some A-packet, such that
$$
\pi\h \tau_1\times\dots\times\tau_k\rtimes\pi_0.
$$
Then $\pi_0$ will be called a primitive representation of A-type.
\end{definition}

We will very briefly recall  the notion of automorphic dual, which  L. Clozel introduced in \cite{MR2331344} (one can find  more details and further references in the original Clozel paper). 

We first recall of notion of the support of a unitary representation $\Pi$ of a locally compact group $\bf G$.  An irreducible unitary representation $\pi$ of $\bf G$ is weakly contained in $\Pi$ if each diagonal matrix coefficient of $\pi$  on each compact subset of $\bf G$ can be approximated by finite sums of diagonal matrix coefficients of $\Pi$ (i.e. each diagonal matrix coefficient of $\pi$  on each compact subset of $\bf G$ is  the limit of sums of diagonal matrix coefficients of $\Pi$). The support of $\Pi$ is the set of equivalence classes of all irreducible unitary representations $\pi$ of $\bf G$ which are weakly contained in $\Pi$.

Let $G$ be a  reductive group defined over an algebraic number
field $k$ (or more generally, over a global field $k$). Fix any place $v$  of $k$ and denote by $k_v$ the completion of $k$ at $v$. The automorphic dual $\widehat G_{v,aut}$ is  the support of 
the representations of the group $G(k_v)$ of $k_v$-rational points of $G$  in the
space of square integrable automorphic forms $L^2(G(k)\backslash H(\mathbb A_k))$ (by right translations). We  denote $F = k_v$.

Motivated by \cite{MR2742213}, we ask the following 

\medskip

\noindent
{\bf Question 9.3.}
Is each primitive representation of A-type isolated in the automorphic dual?

\bibliographystyle{amsalpha}
\bibliography{simple-arthur}

\providecommand{\bysame}{\leavevmode\hbox to3em{\hrulefill}\thinspace}
\providecommand{\MR}{\relax\ifhmode\unskip\space\fi MR }
\providecommand{\MRhref}[2]{%
  \href{http://www.ams.org/mathscinet-getitem?mr=#1}{#2}
}
\providecommand{\href}[2]{#2}
\begin{thebibliography}{Tad98b}

\bibitem[AM20]{AtMiZel}
Hiraku Atobe and Alberto M\'{i}nguez, \emph{The explicit {Z}elevinsky-{A}ubert
  duality}, preprint (2020), \url{https://arxiv.org/pdf/2008.05689.pdf}.

\bibitem[Art13]{MR3135650}
James Arthur, \emph{The endoscopic classification of representations}, American
  Mathematical Society Colloquium Publications, vol.~61, American Mathematical
  Society, Providence, RI, 2013, Orthogonal and symplectic groups. \MR{3135650}

\bibitem[Aub95]{MR1285969}
Anne-Marie Aubert, \emph{Dualit\'e dans le groupe de {G}rothendieck de la
  cat\'egorie des repr\'esentations lisses de longueur finie d'un groupe
  r\'eductif {$p$}-adique}, Trans. Amer. Math. Soc. \textbf{347} (1995), no.~6,
  2179--2189. \MR{1285969}

\bibitem[BBK18]{MR3769724}
Joseph Bernstein, Roman Bezrukavnikov, and David Kazhdan,
  \emph{Deligne-{L}usztig duality and wonderful compactification}, Selecta
  Math. (N.S.) \textbf{24} (2018), no.~1, 7--20. \MR{3769724}

\bibitem[Clo07]{MR2331344}
L.~Clozel, \emph{Spectral theory of automorphic forms}, Automorphic forms and
  applications, IAS/Park City Math. Ser., vol.~12, Amer. Math. Soc.,
  Providence, RI, 2007, pp.~43--93. \MR{2331344 (2008i:11069)}

\bibitem[Jan14]{MR3268853}
Chris Jantzen, \emph{Tempered representations for classical {$p$}-adic groups},
  Manuscripta Math. \textbf{145} (2014), no.~3-4, 319--387. \MR{3268853}

\bibitem[Jan18]{MR3868005}
\bysame, \emph{Duality for classical {$p$}-adic groups: the half-integral
  case}, Represent. Theory \textbf{22} (2018), 160--201. \MR{3868005}

\bibitem[LMT04]{MR2046512}
Erez Lapid, Goran Mui{\'c}, and Marko Tadi{\'c}, \emph{On the generic unitary
  dual of quasisplit classical groups}, Int. Math. Res. Not. (2004), no.~26,
  1335--1354. \MR{2046512 (2005b:22021)}

\bibitem[M{\oe}g02]{MR1913095}
Colette M{\oe}glin, \emph{Sur la classification des s\'{e}ries discr\`etes des
  groupes classiques {$p$}-adiques: param\`etres de {L}anglands et
  exhaustivit\'{e}}, J. Eur. Math. Soc. (JEMS) \textbf{4} (2002), no.~2,
  143--200. \MR{1913095}

\bibitem[M{\oe}g06]{MR2209850}
\bysame, \emph{Sur certains paquets d'{A}rthur et involution
  d'{A}ubert-{S}chneider-{S}tuhler g\'{e}n\'{e}ralis\'{e}e}, Represent. Theory
  \textbf{10} (2006), 86--129. \MR{2209850}

\bibitem[M{\oe}g09]{MR2533005}
\bysame, \emph{Paquets d'{A}rthur discrets pour un groupe classique
  {$p$}-adique}, Automorphic forms and {$L$}-functions {II}. {L}ocal aspects,
  Contemp. Math., vol. 489, Amer. Math. Soc., Providence, RI, 2009,
  pp.~179--257. \MR{2533005}

\bibitem[M{\oe}g11]{MR2767522}
\bysame, \emph{Multiplicit\'{e} 1 dans les paquets d'{A}rthur aux places
  {$p$}-adiques}, On certain {$L$}-functions, Clay Math. Proc., vol.~13, Amer.
  Math. Soc., Providence, RI, 2011, pp.~333--374. \MR{2767522}

\bibitem[MT02]{MR1896238}
Colette M{\oe}glin and Marko Tadi{\'c}, \emph{Construction of discrete series
  for classical {$p$}-adic groups}, J. Amer. Math. Soc. \textbf{15} (2002),
  no.~3, 715--786 (electronic). \MR{1896238 (2003g:22020)}

\bibitem[MT11]{MR2767523}
Goran Mui{\'c} and Marko Tadi{\'c}, \emph{Unramified unitary duals for split
  classical {$p$}-adic groups; the topology and isolated representations}, On
  certain {$L$}-functions, Clay Math. Proc., vol.~13, Amer. Math. Soc.,
  Providence, RI, 2011, pp.~375--438. \MR{2767523 (2012j:22021)}

\bibitem[MW06]{MR2305609}
Colette M{\oe}glin and Jean-Loup Waldspurger, \emph{Sur le transfert des traces
  d'un groupe classique {$p$}-adique \`a un groupe lin\'{e}aire tordu}, Selecta
  Math. (N.S.) \textbf{12} (2006), no.~3-4, 433--515. \MR{2305609}

\bibitem[Rod82]{MR689531}
Fran{\c{c}}ois Rodier, \emph{Repr\'esentations de {${\rm GL}(n,\,k)$} o\`u
  {$k$} est un corps {$p$}-adique}, Bourbaki {S}eminar, {V}ol. 1981/1982,
  Ast\'erisque, vol.~92, Soc. Math. France, Paris, 1982, pp.~201--218.
  \MR{689531 (84h:22040)}

\bibitem[SS97]{MR1471867}
Peter Schneider and Ulrich Stuhler, \emph{Representation theory and sheaves on
  the {B}ruhat-{T}its building}, Inst. Hautes \'Etudes Sci. Publ. Math. (1997),
  no.~85, 97--191. \MR{MR1471867 (98m:22023)}

\bibitem[Tad86]{MR0870688}
Marko Tadi{\'c}, \emph{Classification of unitary representations in irreducible
  representations of general linear group (non-{A}rchimedean case)}, Ann. Sci.
  \'Ecole Norm. Sup. (4) \textbf{19} (1986), no.~3, 335--382. \MR{870688
  (88b:22021)}

\bibitem[Tad87]{MR0894588}
\bysame, \emph{Topology of unitary dual of non-{A}rchimedean {${\rm GL}(n)$}},
  Duke Math. J. \textbf{55} (1987), no.~2, 385--422. \MR{894588 (89c:22029)}

\bibitem[Tad95]{MR1356358}
\bysame, \emph{Structure arising from induction and {J}acquet modules of
  representations of classical {$p$}-adic groups}, J. Algebra \textbf{177}
  (1995), no.~1, 1--33. \MR{1356358 (97b:22023)}

\bibitem[Tad98a]{MR1658535}
\bysame, \emph{On reducibility of parabolic induction}, Israel J. Math.
  \textbf{107} (1998), 29--91. \MR{1658535 (2001d:22012)}

\bibitem[Tad98b]{MR1600280}
\bysame, \emph{On regular square integrable representations of {$p$}-adic
  groups}, Amer. J. Math. \textbf{120} (1998), no.~1, 159--210. \MR{1600280
  (99h:22026)}

\bibitem[Tad99]{MR1698200}
\bysame, \emph{Square integrable representations of classical {$p$}-adic groups
  corresponding to segments}, Represent. Theory \textbf{3} (1999), 58--89
  (electronic). \MR{1698200 (2000d:22020)}

\bibitem[Tad09]{MR2504024}
\bysame, \emph{On reducibility and unitarizability for classical {$p$}-adic
  groups, some general results}, Canad. J. Math. \textbf{61} (2009), no.~2,
  427--450. \MR{2504024 (2010a:22026)}

\bibitem[Tad10]{MR2742213}
\bysame, \emph{On automorphic duals and isolated representations; new
  phenomena}, J. Ramanujan Math. Soc. \textbf{25} (2010), no.~3, 295--328.
  \MR{2742213 (2012a:11061)}

\bibitem[Tad12]{MR2908042}
\bysame, \emph{Reducibility and discrete series in the case of classical
  {$p$}-adic groups; an approach based on examples}, Geometry and analysis of
  automorphic forms of several variables, Ser. Number Theory Appl., vol.~7,
  World Sci. Publ., Hackensack, NJ, 2012, pp.~254--333. \MR{2908042}

\bibitem[Tad18]{MR3969882}
\bysame, \emph{On unitarizability in the case of classical p-adic groups},
  Geometric aspects of the trace formula, Simons Symp., Springer, Cham, 2018,
  pp.~405--453. \MR{3969882}

\bibitem[Tad20]{MR-T-CR3}
\bysame, \emph{Unitarizability in corank three for classical $p$-adic groups},
  Mem. Amer. Math. Soc., to appear (2020).

\bibitem[Xu17a]{MR3679701}
Bin Xu, \emph{On {M}\oe glin's parametrization of {A}rthur packets for
  {$p$}-adic quasisplit {$Sp(N)$} and {$SO(N)$}}, Canad. J. Math. \textbf{69}
  (2017), no.~4, 89--960. \MR{3679701}

\bibitem[Xu17b]{MR3713922}
\bysame, \emph{On the cuspidal support of discrete series for {$p$}-adic
  quasisplit {$Sp(N)$} and {$SO(N)$}}, Manuscripta Math. \textbf{154} (2017),
  no.~3-4, 441--502. \MR{3713922}

\bibitem[Zel80]{MR584084}
Andrei~V. Zelevinsky, \emph{Induced representations of reductive {${\germ
  p}$}-adic groups. {II}. {O}n irreducible representations of {${\rm GL}(n)$}},
  Ann. Sci. \'Ecole Norm. Sup. (4) \textbf{13} (1980), no.~2, 165--210.
  \MR{584084 (83g:22012)}

\end{thebibliography}

\end{document}